\RequirePackage{fix-cm}
\documentclass[smallextended]{svjour3}       
\smartqed  
\usepackage{lipsum}
\usepackage{amsfonts}
\usepackage{graphicx}
\usepackage{epstopdf}
\usepackage{graphicx}
\usepackage{amsmath,multirow}
\usepackage{amssymb}
\usepackage{amscd}
\usepackage{amsfonts}
\usepackage{enumerate}
\usepackage{mathrsfs}
\usepackage{xypic}
\usepackage{stmaryrd}
\usepackage{graphicx}
\usepackage{fancybox}
\usepackage{portland}
\usepackage{color}
\usepackage{amsmath}
\usepackage{exscale}
\usepackage{enumitem,array}%
\usepackage{amssymb}
\usepackage{amsmath}
\usepackage{amssymb}
\usepackage{amscd}
\usepackage{amsfonts}
\usepackage{enumerate}
\usepackage{mathrsfs}
\usepackage{xy}
\usepackage{stmaryrd}
\usepackage{graphicx}
\usepackage{fancybox}
\usepackage{color}
\usepackage{multirow}
\usepackage{exscale}
\usepackage{enumitem,array}%
\usepackage{subfigure}
\usepackage{extarrows}
\usepackage{amsfonts}
\usepackage{xypic}
\usepackage{tikz}
\usepackage{verbatim}
\usepackage{graphicx}
\usetikzlibrary{arrows,chains,matrix,positioning,scopes}
\makeatletter
\tikzset{join/.code=\tikzset{after node path={%
\ifx\tikzchainprevious\pgfutil@empty\else(\tikzchainprevious)%
edge[every join]#1(\tikzchaincurrent)\fi}}}
\makeatother
\tikzset{>=stealth',every on chain/.append style={join},
         every join/.style={->}}
\tikzstyle{labeled}=[execute at begin node=$\scriptstyle,
   execute at end node=$]

\def\pdt2{\partial_t^2}
\def\pdx2{\partial_x^2}
\newcommand{\norm}[1]{\left\Vert#1\right\Vert}
\newcommand{\normmm}[1]{{\left\vert\kern-0.25ex\left\vert\kern-0.25ex\left\vert #1
    \right\vert\kern-0.25ex\right\vert\kern-0.25ex\right\vert}}

\def\pdt2{\partial_t^2}
\def\pdx2{\partial_x^2}

\newcommand{\abs}[1]{\left\vert#1\right\vert}

\DeclareMathOperator{\sinc}{sinc}

\usepackage{url}
\newtheorem{mydef}{Definition}[section]
\newtheorem{mytheo}{Theorem}[section]
\newtheorem{myprop}{Proposition}[section]

\newtheorem{mylemma}{Lemma}[section]

\begin{document}

\title{
Construction and analysis of an explicit  fourth-order symmetric low-regularity trigonometric integrator for the semilinear Klein-Gordon equation
 }

\titlerunning{Fourth-order low-regularity trigonometric integrator}        

\author{Bin Wang
\and Zhen Miao \and Yaolin Jiang
}


\institute{Bin Wang 
\at
              School of Mathematics and Statistics, Xi'an Jiaotong
University, Xi'an, Shannxi 710049, P.R.China\\
              \email{wangbinmaths@xjtu.edu.cn}           
               \and
           Zhen Miao \at
            School of Mathematics and Statistics,
          Northwestern Polytechnical University,
         Xi'an,
Shannxi  710072, P.R. China\\
               \email{linknwxa@nwpu.edu.cn}
 \and
           Yaolin Jiang \at
            School of Mathematics and Statistics, Xi'an Jiaotong
University, Xi'an,
Shannxi 710049, P.R.China\\
               \email{yljiang@mail.xjtu.edu.cn}}

\date{Received: date / Accepted: date}

\maketitle

\begin{abstract}
In this paper, we propose and analyze a novel explicit fourth-order symmetric low-regularity trigonometric integrator for the semilinear Klein--Gordon equation in $d$ dimensions ($d=1,2,3$) with non-smooth solutions. The integrator is constructed by fully exploiting Duhamel's formula, employing a twisted function tailored for trigonometric integrals, and applying a symmetrization procedure. A rigorous error analysis demonstrates that the proposed scheme achieves fourth-order accuracy in the energy space under low regularity requirements, specifically in $H^{2+\mu}(\mathbb{T}^d)\times H^{1+\mu}(\mathbb{T}^d)$ with $	
				\mu=
				\begin{cases}
					1, & d=1,\\
					d/2+\epsilon, & d=2,3,
				\end{cases}$ and $\epsilon>0$. Numerical experiments illustrate that the proposed method offers significantly improved accuracy compared to well-known exponential integrators when approximating the Klein--Gordon equation with non-smooth solutions.
\end{abstract}

\section{Introduction}\label{sec:introduction}
This paper is concerned with the numerical solution of the following semilinear Klein--Gordon equation (SKGE):
\begin{equation}\label{klein-gordon}
\left\{
\begin{aligned}
&\partial_{tt}u(t,x)-\Delta u(t,x)+\rho u(t,x)=f\big(u(t,x)\big),\quad 0< t\leq T,\ x\in \mathbb{T}^d\subset \mathbb{R}^d,\\
&u(0,x)=u_0(x),\ \partial_{t}u(0,x)=v_0(x), \ x\in \mathbb{T}^d,
\end{aligned}
\right.
\end{equation}
subject to periodic boundary conditions. Here, $u(t,x)$ represents the wave displacement at time $t$ and position $x$, $d \in {1,2,3}$ denotes the spatial dimension, $\rho \geq 0$ is a parameter interpreted physically as the particle mass, and $f(u): \mathbb{R} \rightarrow \mathbb{R}$ is a given nonlinear function. Theoretically, the SKGE is globally well-posed for initial data $(u_0, v_0) \in H^{\nu}(\mathbb{T}^d)\times H^{\nu-1}(\mathbb{T}^d)$ with $\nu \geq 1/3$ in one spatial dimension \cite{Bourgain} and $\nu \geq 1$ in two and three dimensions \cite{Ginibre}.
Several physical models arise from \eqref{klein-gordon} through specific choices of $f(u)$. In the linear case $f(u)=0$, the equation is known as the linear Klein--Gordon equation, which describes spinless relativistic composite particles such as pions and Higgs bosons \cite{Davydov76}. For $f(u)=\sin(u)$, \eqref{klein-gordon} reduces to the sine--Gordon equation, which models a wide range of physical phenomena, including dislocation propagation in crystals, the transmission of ultra-short optical pulses in two-level systems, and magnetic flux dynamics in Josephson junctions \cite{Barone71}. Finally, choosing $f(u)=\lambda u^3$ yields the cubic Klein--Gordon equation, where the parameter $\lambda$ (positive for defocusing and negative for focusing interactions) governs the nonlinearity. In the limit $\rho=0$, this system reduces to the cubic wave equation, which finds extensive application in plasma physics \cite{Bellan14}.

Classical time discretization methods for the SKGE have been extensively studied in recent decades. Prominent examples include splitting methods \cite{Bao22,Buchholz2018,Faou15}, trigonometric and exponential integrators \cite{Bao13,Buchholz2021,Gauckler15,Lubich99,WIW}, uniformly accurate methods \cite{Bao14,PI1,S2,Chartier15,WZ21}, and structure-preserving methods \cite{Cano06,Chen20,Cohen08-1,LI20,zhao18}. The analysis of these classical methods typically necessitates strong regularity assumptions. This stems from the fact that, under the condition $f(u) \in L^{\infty}$, two temporal derivatives of the solution generally correspond to two spatial derivatives. Consequently, to achieve an $m$th-order accurate approximation of $(u, \partial_t u)$ in $H^{\nu}(\mathbb{T}^d)\times H^{\nu-1}(\mathbb{T}^d)$ with $\nu\geq d/2$, the initial data are generally required to reside in the stricter space $H^{\nu+m-1}(\mathbb{T}^d)\times H^{\nu+m-2}(\mathbb{T}^d)$. Thus, higher-order schemes impose increasingly stringent regularity requirements. However, such conditions are often not met in applications where initial data may exhibit limited regularity; for instance, high-intensity laser pulses exciting Langmuir waves in plasmas create non-smooth initial conditions for the electric field \cite{Borovskii1993}.

To address this limitation, significant attention has recently been directed toward equations with non-smooth initial data. Pioneering work for the KdV equation includes \cite{A5}. Another seminal contribution is \cite{O18}, which introduced a low-regularity exponential-type integrator for the nonlinear Schr\"{o}dinger equation; this scheme achieves first-order convergence in $H^{\nu}$ for initial data in $H^{\nu+1}$, relaxing the previous $H^{\nu+2}$ requirement. Building on \cite{A5,O18}, various low-regularity integrators have been developed for a wide range of equations, including the Navier–Stokes equation \cite{LI22}, Dirac equation \cite{S21}, KdV equation \cite{R22,Zhao22,LI261}, Schr\"{o}dinger equation \cite{feng23,feng25,O18,O21,O22,R21,LI262}, Boussinesq equation \cite{SU23}, Zakharov system \cite{SU24}, and stochastic PDEs \cite{Cao26}. For a broader theoretical framework, the studies \cite{A3,A1,A6,A4} introduced a general class of low-regularity integrators applicable to nonlinear PDEs, including the sine-Gordon equation.

For the Klein--Gordon system \eqref{klein-gordon}, several significant low-regularity integrators have been developed in recent years. First- and second-order methods were presented in \cite{R21}, where the second-order scheme was shown to converge in the energy space $H^1(\mathbb{T}^d)\times L^2(\mathbb{T}^d)$ under the weak regularity condition $(u_0, v_0) \in H^{1+\frac{d}{4}}(\mathbb{T}^d)\times H^{\frac{d}{4}}(\mathbb{T}^d)$. For the cubic nonlinearity $f(u)=u^3$ in one dimension ($d=1$), a symmetric low-regularity integrator was derived in \cite{Zhao23}, which achieves second-order accuracy in $H^{\nu}(\mathbb{T}) \times H^{\nu-1}(\mathbb{T})$ for $\nu>1/2$ without sacrificing regularity. More recently, for the general SKGE \eqref{klein-gordon} with a nonlinearity $f(u)$, a new integrator was proposed in \cite{LI23} achieving second-order convergence in $H^1(\mathbb{T}^d)\times L^2(\mathbb{T}^d)$ under the regularity condition $(u_0, v_0) \in H^{1+\frac{d}{4}}(\mathbb{T}^d)\times H^{\frac{d}{4}}(\mathbb{T}^d)$ for $d=1,2,3$. While requiring the same regularity as the method in \cite{R21}, the scheme in \cite{LI23} features a significantly simpler formulation.

To the best of our knowledge, the majority of existing low-regularity integrators for the Klein--Gordon equation—as well as for related equations—are limited to first- and second-order schemes. \emph{A crucial aspect is to construct concrete higher-order algorithms that preserve low-regularity properties while maintaining a compact formulation suitable for direct implementation.}
 Although the formalism of decorated trees \cite{A3,A1,A6,A4} confirms the theoretical existence of such high-order integrators, the intricate details involved in their construction have so far hindered the derivation of explicit schemes beyond second order. Generally, the design of concrete high-order low-regularity integrators faces two primary challenges: controlling spatial derivatives within the approximation to effectively relax regularity requirements, and managing the inherent structural complexity of the design process, which imposes substantial constraints on the formulation of higher-order schemes.

This paper is devoted to the explicit construction and rigorous analysis of a fourth-order low-regularity integrator. Our primary objective is to provide a streamlined derivation and formulation. Specifically, we design and analyze a fourth-order low-regularity integrator  for the semilinear Klein--Gordon equation \eqref{klein-gordon} with a general nonlinearity $f(u)$ in dimensions $d=1,2,3$. Based on Duhamel's formula, the proposed method is characterized by its structural simplicity.  The construction proceeds as follows: we first embed the structure of the SKGE into the formulation and subsequently apply a twisted function to the trigonometric integrals arising from the Duhamel formulation. By strategically selecting tractable terms from these integrals (see the approximations for $I^u$ and $I^v$ in Section \ref{sds-construct}), spatial derivatives are optimally distributed across the product terms in the remainder, yielding a third-order scheme. Finally, a symmetrization procedure is employed to upgrade the scheme to fourth order while preserving the low-regularity properties.
We prove that the proposed integrator achieves fourth-order convergence in the energy space $H^1(\mathbb{T}^d)\times L^2(\mathbb{T}^d)$ under the regularity condition
$$(u_0(x),v_0(x))\in H^{2+ \mu}(\mathbb{T}^d)\times H^{1+ \mu}(\mathbb{T}^d)$$
with $	
				\mu=
				\begin{cases}
					1, & d=1,\\
					d/2+\epsilon, & d=2,3,
				\end{cases}$ and $\epsilon>0$. The resulting scheme reads
 \begin{equation}\label{method-sec1}
\begin{aligned}
u_{1}=&\cos(h \sqrt{\mathcal{A}}) u_0+h  \sinc( h \sqrt{\mathcal{A}})  v_0+h^2\Phi_1(h \sqrt{\mathcal{A}}) f( u_0)\\&+h^3  \Psi_1(h \sqrt{\mathcal{A}})f'( u_0)v_0 +h^4  \Psi_2(h \sqrt{\mathcal{A}}) F_1(u_0,v_0),\\
v_{1}=&-h \mathcal{A} \sinc( h \sqrt{\mathcal{A}}) u_0+\cos(h \sqrt{\mathcal{A}}) v_0+
h  \Phi_2(h \sqrt{\mathcal{A}})  f( u_0) \\&+h^2\Phi_1(h \sqrt{\mathcal{A}})f'( u_0)v_0 +h^3  \Psi_1(h \sqrt{\mathcal{A}}) \big(F_1(u_0,v_0)+f'(u_0) f(u_0) \big),\\
u_{n+1}=&\cos(2h \sqrt{\mathcal{A}}) u_{n-1}+2h  \sinc( 2h \sqrt{\mathcal{A}})  v_{n-1}\\
&+ h^2\Big(\big(1-\cos(2h \sqrt{\mathcal{A}})\big)\Phi_1(h \sqrt{\mathcal{A}})+2\mathcal{A} \sinc(2 h \sqrt{\mathcal{A}})  \Phi_2(h \sqrt{\mathcal{A}}) \Big)f( u_n)\\
&+h^3 \Big(\big(1+\cos(2h \sqrt{\mathcal{A}})\big) \Psi_1(h \sqrt{\mathcal{A}})-2 \mathcal{A} \sinc(2 h \sqrt{\mathcal{A}})  \Phi_1(h \sqrt{\mathcal{A}})\Big) f'( u_n)v_n \\
&+h^4\Big( \big(1-\cos(2h \sqrt{\mathcal{A}})\big) \Psi_2(h \sqrt{\mathcal{A}})+2 \mathcal{A} \sinc(2 h \sqrt{\mathcal{A}})    \Psi_1(h \sqrt{\mathcal{A}}) \Big)  F_1(u_n,v_n)\\
&+2h^4 \mathcal{A} \sinc(2 h \sqrt{\mathcal{A}})     \Psi_1(h \sqrt{\mathcal{A}})  f'(u_n) f(u_n),  \\
v_{n+1}
   =&-2h \mathcal{A} \sinc(2 h \sqrt{\mathcal{A}}) u_{n-1}+\cos(2h \sqrt{\mathcal{A}}) v_{n-1}\\
   &+h \Big(\big(1+\cos(2h \sqrt{\mathcal{A}})\big) \Phi_2(h \sqrt{\mathcal{A}}) +2h^2 \mathcal{A} \sinc(2 h \sqrt{\mathcal{A}})  \Phi_1(h \sqrt{\mathcal{A}})\Big)  f( u_n) \\
   &+h^2\Big(\big(1-\cos(2h \sqrt{\mathcal{A}})\big)\Phi_1(h \sqrt{\mathcal{A}})-2h^2 \mathcal{A} \sinc(2 h \sqrt{\mathcal{A}})    \Psi_1(h \sqrt{\mathcal{A}})\Big) f'( u_n)v_n  \\
   &+h^3\Big(\big(1+\cos(2h \sqrt{\mathcal{A}})\big)  \Psi_1(h \sqrt{\mathcal{A}}) +2h^2 \mathcal{A} \sinc(2 h \sqrt{\mathcal{A}})   \Psi_2(h \sqrt{\mathcal{A}})\Big)  F_1(u_n,v_n) \\
      &+h^3 \big(1+\cos(2h \sqrt{\mathcal{A}})\big)  \Psi_1(h \sqrt{\mathcal{A}})   f'(u_n) f(u_n), \ \ n=1,2,\ldots, T/h-1,
\end{aligned}
\end{equation}
where  $h$ is the time stepsize, 
$F_1(u_n,v_n)= f''(u_n)(v_n^2-(\nabla u_n)^2)+\rho  f(u_n)-\rho  f'(u_n)u_n,$
 and  the coefficient functions $\Phi_1,\Phi_2,\Psi_1,\Psi_2$ are defined in \eqref{method-coefficient}.

The main contributions of this work are summarized as follows:
\begin{itemize}
\item In contrast to many existing low-regularity methods, including first- and second-order schemes, the proposed integrator \eqref{method-sec1} features a remarkably compact formulation that is readily implementable and accessible to researchers across various disciplines.
\item Furthermore, the algorithm is purely explicit and non-iterative, thereby ensuring minimal computational overhead and high efficiency.
\item We establish rigorous error bounds for the SKGE \eqref{klein-gordon} subject to non-smooth initial data, demonstrating that the integrator is robustly applicable to solutions of low regularity.
\end{itemize}

The remainder of this paper is organized as follows. Section \ref{sec:2} presents the formulation of the semi-discrete and fully discrete integrators. Section \ref{sec:3} provides a rigorous error analysis of the proposed scheme. Section \ref{sec:4} illustrates the accuracy and efficiency of the new integrator through numerical comparisons with existing methods in the literature. Finally, some conclusions are drawn in Section \ref{sec:5}.

\section{Construction of low-regularity  integrator}\label{sec:2}
In this section, we first introduce the notation and technical tools used throughout the paper. Subsequently, we present the new low-regularity integrator, including both semi-discrete and fully discrete schemes, and detail its construction in Subsection \ref{sds-construct}.
\subsection{Notations and technical tools}
For simplicity, we write $A\lesssim B$ to denote $A\leq CB$ for some generic constant $C>0$. This constant $C$ may depend on the final time $T$ and the norms $\norm{u_0}_{H^{2+ \mu}}$, $\norm{v_0}_{H^{ 1+\mu}}$, and it may vary between occurrences, but it is always independent of the number of discrete points and the time/space step sizes.

For the linear differential operator appearing in \eqref{klein-gordon}, we introduce the concise notation \begin{equation}\label{notations A}
(\mathcal{A}w)(x):=-\Delta  w(x)+ \rho w(x)\end{equation} with $w(x)$ on $\mathbb{T}^d$.
In the construction of the integrator, we will  frequently use the following functions
\begin{equation}\label{notations w}\begin{aligned}&\sinc(t):=\frac{\sin(t)}{t},\ \quad   \qquad \alpha(t):=\left(
              \begin{array}{c}
                \cos(t\sqrt{\mathcal{A}})  \\
                t \sinc( t\sqrt{\mathcal{A}})  \\
              \end{array}
            \right),\\&\beta(t):=\left(
              \begin{array}{c}
                t \sinc( t\sqrt{\mathcal{A}})    \\
                \cos(t\sqrt{\mathcal{A}}) \\
              \end{array}
            \right),\  \gamma(t): =\left(
              \begin{array}{c}
                \sqrt{\mathcal{A}} \sin(t\sqrt{\mathcal{A}})    \\
                \cos(t\sqrt{\mathcal{A}}) \\
              \end{array}
            \right).
\end{aligned}\end{equation}

The Sobolev space of functions on the domain  $\mathbb{T}^d=[a,b]^d$ is considered in the whole paper,  and we shall refer it as  $H^{\nu}(\mathbb{T}^d)$ for any $\nu\geq 0$.
The norm on this Sobolev space    is denoted by
$$\norm{f}^2_{H^{\nu}}=(b-a)^{d}\sum_{\xi\in \mathbb{Z}^d} (1+\abs{\xi}^2)^{\nu}\abs{\widehat{f}(\xi)}^2,$$
where $\widehat{f}$ is the Fourier transform of   $f(x)$ which is defined by
$$\widehat{f}(\xi)=\frac{1}{(b-a)^d}\int_{\mathbb{T}^d} e^{-i x \cdot \xi}f(x) dx$$ with $\xi=(\xi^1,\xi^2,\ldots,\xi^d)\in \mathbb{Z}^d$
and $x=(x^1,x^2,\ldots,x^d)\in \mathbb{Z}^d$. Here $\abs{\cdot}$ is defined as
$$\abs{\xi}=\sqrt{(\xi^1)^2 +(\xi^2)^2+\ldots+(\xi^d)^2}$$ and $$x \cdot \xi=x^1\xi^1+x^2\xi^2+\ldots+x^d\xi^d.$$

For the norm on  Sobolev space,  some obvious properties are  frequently used in this article and we summarize them as follows.
 \begin{itemize}
   \item With the notation $J^{\nu}:=(1-\Delta)^{\frac{\nu}{2}}$ for $\nu\geq 0$, it is clear that
$$
\widehat{ J^{\nu}f}(\xi)=(1+\abs{\xi}^2)^{\frac{\nu}{2}}\widehat{f}(\xi),\ \ \norm{J^{\nu}f}_{L^2}=\norm{f}_{H^{\nu}}.
$$
   \item  For any function $\sigma: \mathbb{Z}^d\rightarrow \mathbb{C}^d$ such that $\abs{\sigma(\xi)}\leq C_{\sigma} (1+\abs{\xi}^2)^m$
with some constants $C_{\sigma}$ and $m$,  the operator  $$\sigma( \sqrt{-\Delta}): H^{\nu}(\mathbb{T}^d)\rightarrow H^{\nu-m}(\mathbb{T}^d)$$ has the following result  for $\nu\geq m$
$$\sigma(\sqrt{-\Delta})f(x)=\sum_{\xi\in \mathbb{Z}^d} \sigma(\xi)  \widehat{f}(\xi)  e^{i x \cdot \xi},\ \ \norm{\sigma(\sqrt{-\Delta})f}^2_{H^{\nu-m}}  \leq  C_{\sigma}  \norm{f}^2_{H^{\nu}}.$$

   \item  For the special  operator $ \sqrt{-\Delta}^{-1} $  which is defined by  $$\widehat{ \sqrt{-\Delta}^{-1}f}(\xi)=\left\{
\begin{split}
&
\frac{\widehat{ f}(\xi)}{\abs{\xi}},\quad
\textmd{when}\ \ \ \abs{\xi}\neq 0,\\
&0, \quad\ \ \ \ \
\textmd{when}\ \ \ \abs{\xi}= 0,
\end{split}\right.$$
it is straightforward to verify that
$\norm{\sqrt{-\Delta}^{-1}f}_{H^{\nu+1}}  \lesssim  \norm{f}_{H^{\nu}}.$
 \end{itemize}

For the operator $\mathcal{A}$  introduced in \eqref{notations A}, the following estimates are consequences of the properties stated above.
\begin{myprop}\label{lamma00}   If   $f\in H^{\nu}$ for any $\nu\geq 0$,  then the following results hold
$$\norm{\cos(t\sqrt{\mathcal{A}}) f}_{H^{\nu}}\leq \norm{f}_{H^{\nu}},\  \norm{\sin(t\sqrt{\mathcal{A}})  f}_{H^{\nu}}\leq \norm{f}_{H^{\nu}},$$
and $$ \norm{\sinc( t\sqrt{\mathcal{A}}) f}_{H^{\nu}}\leq \norm{f}_{H^{\nu}},\ \ \norm{\sqrt{\mathcal{A}}^{-1}f}_{H^{\nu+1}}\lesssim \norm{f}_{H^{\nu}}.$$

\end{myprop}
\begin{proof}The first three inequalities are obvious. For the last one, a direct calculation shows that $\norm{\sqrt{\mathcal{A}}^{-1}f}_{H^{\nu+1}}=\norm{f}_{H^{\nu}}$ if $ \rho=1$,  $\norm{\sqrt{\mathcal{A}}^{-1}f}_{H^{\nu+1}}< \norm{f}_{H^{\nu}}$ if $ \rho>1$, and $\norm{\sqrt{\mathcal{A}}^{-1}f}_{H^{\nu+1}}< 2 \norm{f}_{H^{\nu}}$ if $ \rho<1$. This completes the proof.
\end{proof}

The following version of the Kato–Ponce inequalities will also be needed in this paper, which was originally given in   \cite{Kato88}.

\begin{mylemma}\label{lamma0}(The Kato–Ponce inequalities)
In the regime $\nu> d/4$, we call upon the classical bilinear estimate
$$\norm{fg}_{H^{2 \nu}}\lesssim \norm{f}_{H^{2 \nu}}\norm{g}_{H^{2 \nu}},\ \ \norm{J^{-1}(Jfg)}_{H^{2 \nu}}\lesssim \norm{f}_{H^{2 \nu}}\norm{g}_{H^{2 \nu}},$$
whereas in the regime $0\leq\nu\leq d/4$, the following   bilinear estimates are true
 \begin{equation}\label{KPI}
\begin{aligned}
&\norm{fg}_{H^{2 \nu}}\lesssim \norm{f}_{H^{ \frac{d}{4}+\nu}}\norm{g}_{H^{\frac{d}{4}+\nu}}\ \ \ \textmd{for}\ \ 0\leq\nu< d/4,\\
&\norm{fg}_{H^{2 \nu}}\lesssim \norm{f}_{H^{ \frac{d}{2}+\epsilon}}\norm{g}_{H^{2 \nu}}\ \ \ \ \  \textmd{for}\ \ 0\leq\nu\leq d/4,
\end{aligned}
\end{equation}
with any $\epsilon>0$.
 If $\nu>d/2$ and $f, g \in H^{\nu}$, we have the   inequality $\norm{J^{\nu}(fg)}_{L^{2}}\lesssim \norm{f}_{H^{\nu }}\norm{g}_{H^{\nu}}.$
 If $\delta \geq0, \nu> d/2$ and $f \in H^{\delta+\nu}$, $g \in H^{\delta}$, then
$\norm{J^{\delta}(fg)}_{L^{2}}\lesssim \norm{f}_{H^{\delta+\nu }}\norm{g}_{H^{\delta}}.$
\end{mylemma}

\subsection{Semi-discrete and fully-discrete  integrators}\label{sds}
We first directly present the scheme of semi-discrete integrator  and its deduction will be given in
Subsection \ref{sds-construct}.
\begin{mydef}\label{def-se} (\textbf{Semi-discrete  integrator})
 Let $t_n=n h $ with $n=0,1,\ldots, N$  be a uniform partition of the time interval $[0,T]$
with stepsize  $h=T/N$,  where $N$ is any given positive integer.
For solving the   Klein--Gordon
equation \eqref{klein-gordon}, we define the numerical solution $u_{n}=u_{n}(x)\approx u(t_n,x), v_{n}=v_{n}(x)\approx v(t_n,x):=\partial_tu(t_n,x)$ by the
following  semi-discrete  integrator
 \begin{equation}\label{method}
\begin{aligned}
u_{1}=&\cos(h \sqrt{\mathcal{A}}) u_0+h  \sinc( h \sqrt{\mathcal{A}})  v_0+h^2\Phi_1(h \sqrt{\mathcal{A}}) f( u_0)\\&+h^3  \Psi_1(h \sqrt{\mathcal{A}})f'( u_0)v_0 +h^4  \Psi_2(h \sqrt{\mathcal{A}}) F_1(u_0,v_0),\\
v_{1}=&-h \mathcal{A} \sinc( h \sqrt{\mathcal{A}}) u_0+\cos(h \sqrt{\mathcal{A}}) v_0+
h  \Phi_2(h \sqrt{\mathcal{A}})  f( u_0) \\&+h^2\Phi_1(h \sqrt{\mathcal{A}})f'( u_0)v_0 +h^3  \Psi_1(h \sqrt{\mathcal{A}}) \big(F_1(u_0,v_0)+f'(u_0) f(u_0) \big),\\
u_{n+1}=&\cos(2h \sqrt{\mathcal{A}}) u_{n-1}+2h  \sinc( 2h \sqrt{\mathcal{A}})  v_{n-1}\\
&+ h^2\Big(\big(1-\cos(2h \sqrt{\mathcal{A}})\big)\Phi_1(h \sqrt{\mathcal{A}})+2\mathcal{A} \sinc(2 h \sqrt{\mathcal{A}})  \Phi_2(h \sqrt{\mathcal{A}}) \Big)f( u_n)\\
&+h^3 \Big(\big(1+\cos(2h \sqrt{\mathcal{A}})\big) \Psi_1(h \sqrt{\mathcal{A}})-2 \mathcal{A} \sinc(2 h \sqrt{\mathcal{A}})  \Phi_1(h \sqrt{\mathcal{A}})\Big) f'( u_n)v_n \\
&+h^4\Big( \big(1-\cos(2h \sqrt{\mathcal{A}})\big) \Psi_2(h \sqrt{\mathcal{A}})+2 \mathcal{A} \sinc(2 h \sqrt{\mathcal{A}})    \Psi_1(h \sqrt{\mathcal{A}}) \Big)  F_1(u_n,v_n)\\
&+2h^4 \mathcal{A} \sinc(2 h \sqrt{\mathcal{A}})     \Psi_1(h \sqrt{\mathcal{A}})  f'(u_n) f(u_n),  \\
v_{n+1}
   =&-2h \mathcal{A} \sinc(2 h \sqrt{\mathcal{A}}) u_{n-1}+\cos(2h \sqrt{\mathcal{A}}) v_{n-1}\\
   &+h \Big(\big(1+\cos(2h \sqrt{\mathcal{A}})\big) \Phi_2(h \sqrt{\mathcal{A}}) +2h^2 \mathcal{A} \sinc(2 h \sqrt{\mathcal{A}})  \Phi_1(h \sqrt{\mathcal{A}})\Big)  f( u_n) \\
   &+h^2\Big(\big(1-\cos(2h \sqrt{\mathcal{A}})\big)\Phi_1(h \sqrt{\mathcal{A}})-2h^2 \mathcal{A} \sinc(2 h \sqrt{\mathcal{A}})    \Psi_1(h \sqrt{\mathcal{A}})\Big) f'( u_n)v_n  \\
   &+h^3\Big(\big(1+\cos(2h \sqrt{\mathcal{A}})\big)  \Psi_1(h \sqrt{\mathcal{A}}) +2h^2 \mathcal{A} \sinc(2 h \sqrt{\mathcal{A}})   \Psi_2(h \sqrt{\mathcal{A}})\Big)  F_1(u_n,v_n) \\
      &+h^3 \big(1+\cos(2h \sqrt{\mathcal{A}})\big)  \Psi_1(h \sqrt{\mathcal{A}})   f'(u_n) f(u_n), 
\end{aligned}
\end{equation}
where  $n=1,2,\ldots, T/h-1$, the initial value $u_0:=u_0(x), v_0:=v_0(x)$ is given in \eqref{klein-gordon}, the notation $F_1$ is
$$F_1(u_n,v_n)= f''(u_n)(v_n^2-(\nabla u_n)^2)+\rho  f(u_n)-\rho  f'(u_n)u_n,$$
 and  the coefficient functions of the integrator are defined by
 \begin{equation}\label{method-coefficient}
\begin{aligned}
 &\Phi_1(h \sqrt{\mathcal{A}})= \frac{\sinc( h \sqrt{\mathcal{A}})}{2},\ \ \  \ \ \Psi_2(h \sqrt{\mathcal{A}})= \frac{ 1-\cos( h \sqrt{\mathcal{A}})-\frac{h}{2} \sqrt{\mathcal{A}}\sin(h \sqrt{\mathcal{A}})}{(h \sqrt{\mathcal{A}})^4},\\
  &\Psi_1(h \sqrt{\mathcal{A}})= \frac{ \sinc( h \sqrt{\mathcal{A}})-\cos(h \sqrt{\mathcal{A}})}{2(h \sqrt{\mathcal{A}})^2},   \ \Phi_2(h \sqrt{\mathcal{A}})= \frac{ \cos(h \sqrt{\mathcal{A}}) + \sinc( h \sqrt{\mathcal{A}})}{2}.
\end{aligned}
\end{equation}
Here
the prime on $f$ indicates the   derivative  of $f(u)$ w.r.t. $u$.
 
\end{mydef}

\begin{remark}
From \eqref{method}, it can be observed clearly that this method is a kind of  trigonometric integrators (\cite{hairer2006,Lubich99}). The scheme \eqref{method} is fully explicit and  not complicated even in comparison with some first-order or second-order low-regularity integrators.  Therefore, it is convenient to be implemented in practical computations and has low computation cost. This aspect will be  demonstrated clearly by a numerical test given in Section \ref{sec:4}.
\end{remark}

In what follows, we consider the spatial discretization of \eqref{method}, which can be
handled by the Fourier spectral method \cite{Shen}. For conciseness, we assume that
$\mathbb{T}^d:=[0,1]^d$ is equipped with the periodic boundary condition. Then any function
$W\in H^1_{\mathrm{per}}(\mathbb{T}^d)\times L^2(\mathbb{T}^d)$ can be expanded into the
Fourier series
$$W=\sum_{\mathbf{k}\in\mathbb{Z}^d}\widehat{W}_{\mathbf{k}}\,e^{\,2\pi i\,\mathbf{k}\cdot x},
\qquad
\widehat{W}_{\mathbf{k}}=\int_{\mathbb{T}^d}W(x)\,e^{-2\pi i\,\mathbf{k}\cdot x}\,\mathrm{d}x,$$
where $\mathbf{k}=(k_1,\ldots,k_d)$ and
$\mathbf{k}\cdot x=k_1x_1+\cdots+k_dx_d$. Choose a positive integer $N_x$ and denote by
$I_{N_x}$ and $\Pi_{N_x}$ the Fourier interpolation and $L^2$-orthogonal projection
operators onto $S_{N_x}$, respectively, where
\begin{equation*}
\begin{aligned}
&S_{N_x}=\Big\{\sum_{\mathbf{k}\in\mathbb{K}_{N_x}}\widehat{W}_{\mathbf{k}}\,
e^{\,2\pi i\,\mathbf{k}\cdot x}:\ \widehat{W}_{\mathbf{k}}\in\mathbb{C}\Big\},\\
&\mathbb{K}_{N_x}=\big\{\mathbf{k}\in\mathbb{Z}^d:\ |k_j|\le N_x,\ j=1,\ldots,d\big\}.
\end{aligned}
\end{equation*}
Here, $I_{N_x}W$ is the unique function in $S_{N_x}$ satisfying
$I_{N_x}W(x_{\mathbf{j}})=W(x_{\mathbf{j}})$ at the equidistant grid points
$x_{\mathbf{j}}=(\tfrac{j_1}{2N_x+1},\ldots,\tfrac{j_d}{2N_x+1})$,
$0\le j_k\le 2N_x$, and $(\Pi_{N_x}W-W,\,V)=0$ for any $W\in L^2(\mathbb{T}^d)$
and $V\in S_{N_x}$. Based on the semi-discrete scheme in Definition \ref{def-se},
the fully discrete low-regularity integrator for the periodic Klein--Gordon equation
is defined as follows.

\begin{mydef}\label{def-ful} (\textbf{Fully discrete  integrator})
For the semi-discrete integrator \eqref{method},  the fully discrete low-regularity integrator
is given by
 \begin{equation}\label{method-f}
\begin{aligned}
U_{1}=&\cos(h \sqrt{\mathcal{A}}) U_0+h  \sinc( h \sqrt{\mathcal{A}})  V_0+h^2\Phi_1(h \sqrt{\mathcal{A}})I_{N_x} f( U_0)\\&+h^3  \Psi_1(h \sqrt{\mathcal{A}})I_{N_x} \big(f'( U_0)V_0\big)+h^4  \Psi_2(h \sqrt{\mathcal{A}}) I_{N_x}F_1(U_0,V_0),\\
V_{1}
   =&-h \mathcal{A} \sinc( h \sqrt{\mathcal{A}}) U_0+\cos(h \sqrt{\mathcal{A}}) V_0\\&+
h  \Phi_2(h \sqrt{\mathcal{A}}) I_{N_x} f( U_0) +h^2\Phi_1(h \sqrt{\mathcal{A}})I_{N_x}\big(f'(U_0)V_0\big)  \\&+h^3  \Psi_1(h \sqrt{\mathcal{A}}) I_{N_x}\big(F_1(U_0,V_0)+ f'(U_0)  f(U_0)\big),\\
U_{n+1}=&\cos(2h \sqrt{\mathcal{A}}) U_{n-1}+2h  \sinc( 2h \sqrt{\mathcal{A}})  V_{n-1}\\
&+ h^2\Big(\big(1-\cos(2h \sqrt{\mathcal{A}})\big)\Phi_1(h \sqrt{\mathcal{A}})+2\mathcal{A} \sinc(2 h \sqrt{\mathcal{A}})  \Phi_2(h \sqrt{\mathcal{A}}) \Big)I_{N_x}f( U_n)\\
&+h^3 \Big(\big(1+\cos(2h \sqrt{\mathcal{A}})\big) \Psi_1(h \sqrt{\mathcal{A}})-2 \mathcal{A} \sinc(2 h \sqrt{\mathcal{A}})  \Phi_1(h \sqrt{\mathcal{A}})\Big) 
I_{N_x} \big(f'( U_n)V_n\big)\\
&+h^4\Big( \big(1-\cos(2h \sqrt{\mathcal{A}})\big) \Psi_2(h \sqrt{\mathcal{A}})+2 \mathcal{A} \sinc(2 h \sqrt{\mathcal{A}})    \Psi_1(h \sqrt{\mathcal{A}}) \Big)  I_{N_x}F_1(U_n,V_n)\\
&+2h^4 \mathcal{A} \sinc(2 h \sqrt{\mathcal{A}})     \Psi_1(h \sqrt{\mathcal{A}})  I_{N_x} \big(f'( U_n)f( U_n)\big),  \\
V_{n+1}
   =&-2h \mathcal{A} \sinc(2 h \sqrt{\mathcal{A}}) U_{n-1}+\cos(2h \sqrt{\mathcal{A}}) V_{n-1}\\
   &+h \Big(\big(1+\cos(2h \sqrt{\mathcal{A}})\big) \Phi_2(h \sqrt{\mathcal{A}}) +2h^2 \mathcal{A} \sinc(2 h \sqrt{\mathcal{A}})  \Phi_1(h \sqrt{\mathcal{A}})\Big) 
    I_{N_x}f( U_n) \\
   &+h^2\Big(\big(1-\cos(2h \sqrt{\mathcal{A}})\big)\Phi_1(h \sqrt{\mathcal{A}})-2h^2 \mathcal{A} \sinc(2 h \sqrt{\mathcal{A}})    \Psi_1(h \sqrt{\mathcal{A}})\Big) I_{N_x} \big(f'( U_n)V_n\big)\\
   &+h^3\Big(\big(1+\cos(2h \sqrt{\mathcal{A}})\big)  \Psi_1(h \sqrt{\mathcal{A}}) +2h^2 \mathcal{A} \sinc(2 h \sqrt{\mathcal{A}})   \Psi_2(h \sqrt{\mathcal{A}})\Big) 
   I_{N_x} F_1(U_n,V_n) \\
      &+h^3 \big(1+\cos(2h \sqrt{\mathcal{A}})\big)  \Psi_1(h \sqrt{\mathcal{A}})   I_{N_x} \big(f'( U_n)f( U_n)\big), 
\end{aligned}
\end{equation}
where  $0<h <1$ is the time stepsize, $n=1,2,\ldots, T/h-1$ and the initial values are chosen as $U_0=\Pi_{N_x} u(x)$ and  $V_0=\Pi_{N_x} v(x)$ for
  $x\in \{\frac{2n}{2N_x+1}: n=1,2,\ldots,N_x \}^d$. In practical computation,  the trigonometric interpolation $I_{N_x}$ can be implemented with Fast Fourier Transform (FFT).
\end{mydef}

\subsection{Construction of semi-discrete integrator}\label{sds-construct}
In this subsection, we present the construction of the semi-discrete numerical method \eqref{method} based on twisted functions and
Duhamel’s formula.  For readers’ convenience,    the $x$-dependence of the unknown
functions is omitted and some technical
estimates of remainders are deferred to   Section \ref{sec:3}, where a rigorous error analysis will be given.

 For   the nonlinear Klein--Gordon
equation \eqref{klein-gordon},  the Duhamel’s formula at $t=t_n+s$ with $s\in \mathbb{R}$ reads
\begin{equation}\label{Duhamel}
\begin{aligned}
u(t_n+s)=&\cos(s\sqrt{\mathcal{A}}) u(t_n)+s \sinc( s\sqrt{\mathcal{A}})  v(t_n)\\&+\int_0^{s } (s-\theta)\sinc( (s-\theta)\sqrt{\mathcal{A}}) f\big( u(t_n+\theta)\big)d\theta,\\
v(t_n+s )=&-s\mathcal{A} \sinc( s\sqrt{\mathcal{A}}) u(t_n)+\cos(s\sqrt{\mathcal{A}})  v(t_n)\\&+\int_0^{s }  \cos( (s-\theta)\sqrt{\mathcal{A}}) f\big( u(t_n+\theta)\big)d\theta,
\end{aligned}
\end{equation}
which plays a crucial role in the method's formulation. The key procedures for the algorithm construction are outlined as follows.
\begin{itemize}
  \item \textbf{Step 1.} Using this formula, we insert the expression of $u(t_n+\theta)$  into  the trigonometric integrals on the right hand side of \eqref{Duhamel} and deduce an integral formulation of $f\big(u(t_n+ s)\big)$. 
  \item \textbf{Step 2.} Then by carefully selecting
the tractable terms from the obtained formulae and dropping some parts which do not affect third-order accuracy and regularity, a transitional third-order   semi-discrete scheme is formulated. 
  \item \textbf{Step 3.} Finally, based on this  transitional scheme and a symmetrization procedure, we upgrade the scheme to fourth order while retaining the low-regularity requirements.
\end{itemize}

\subsubsection{Step 1. An integral formulation of $f\big(u(t_n+ s)\big)$}

Firstly, we get the expression of $u(t_n+s)$ by the first equation of \eqref{Duhamel} and then insert this into
$f\big(u(t_n+ s)\big)$. This gives
\begin{equation}\label{funs}
\begin{aligned}
f\big(u(t_n+ s)\big)=&f\Big(\alpha^{\intercal}(s)W(t_n)\Big)+f'\Big(\alpha^{\intercal}(s)W(t_n)\Big)
\Big(u(t_n+ s)-\alpha^{\intercal}(s)W(t_n)\Big)\\
&+R_{f''}(t_n, s)\Big(u(t_n+ s)-\alpha^{\intercal}(s)W(t_n)\Big)^2,
\end{aligned}
\end{equation}
where $W(t_n)=(
             u(t_n),
             v(t_n) )^{\intercal}$   and \begin{equation}\label{RF2}\begin{aligned}&R_{f''}
=\int_0^1 \int_0^1\theta f''\Big( (1-\xi)\alpha^{\intercal}(s)W(t_n)+
\xi(1-\theta) \alpha^{\intercal}(s)W(t_n)+\theta u(t_n+ s)\Big)d\xi d\theta.\end{aligned}\end{equation}
Using the expression of $u(t_n+s)$ again in the right hand side of \eqref{funs}   yields the following expression:
\begin{equation}\label{funs-new}
\begin{aligned}
&f\big(u(t_n+ s)\big)\\
=&f\Big(\alpha^{\intercal}(s)W(t_n)\Big)
+f'\Big(\alpha^{\intercal}(s)W(t_n)\Big)
\int_0^{s }  \frac{\sin( ( s-\theta)\sqrt{\mathcal{A}})}{\sqrt{\mathcal{A}}} f\big( u(t_n+\theta)\big)d\theta \\
&+R_{f''}(t_n, s) \Big(\int_0^{s }  \frac{\sin( ( s-\theta)\sqrt{\mathcal{A}})}{\sqrt{\mathcal{A}}} f\big( u(t_n+\theta)\big)d\theta\Big)^2.
\end{aligned}
\end{equation}

Now we obtained a   desired form of $f\big(u(t_n+ s)\big)$ and based on which, the two trigonometric integrals  on the right hand side of \eqref{Duhamel} can be reformulated as
 \begin{equation}\label{local errors 1-new}
\begin{aligned}
&\int_0^{h }  \frac{\sin( (h -s)\sqrt{\mathcal{A}}) }{\sqrt{\mathcal{A}}} f\big( u(t_n+s)\big)ds
=\underbrace{\int_0^{h  }  \frac{\sin( (h -s)\sqrt{\mathcal{A}}) }{\sqrt{\mathcal{A}}}f\Big(\alpha^{\intercal}(s)W(t_n)\Big)ds}_{=:\textmd{  I}^{u}}\\
&+\underbrace{\int_0^{h  } \frac{\sin( (h -s)\sqrt{\mathcal{A}}) }{\sqrt{\mathcal{A}}}f'\Big(\alpha^{\intercal}(s)W(t_n)\Big) \int_0^{s }  \frac{\sin( ( s-\theta)\sqrt{\mathcal{A}})}{\sqrt{\mathcal{A}}} f\big( u(t_n+\theta)\big)d\theta ds}_{=:\textmd{  II}^{u}}
\\
&+\underbrace{\int_0^{h  }  \frac{\sin( (h -s)\sqrt{\mathcal{A}}) }{\sqrt{\mathcal{A}}}R_{f''}(t_n, s) \Big(\int_0^{s }  \frac{\sin( ( s-\theta)\sqrt{\mathcal{A}})}{\sqrt{\mathcal{A}}} f\big( u(t_n+\theta)\big)d\theta \Big)^2 ds}_{=:\textmd{  III}^{u}},
\end{aligned}
\end{equation}
and
 \begin{equation}\label{local errorsv 1-new}
\begin{aligned}
&\int_0^{h } \cos( (h-\theta)\sqrt{\mathcal{A}}) f\big( u(t_n+\theta)\big)d\theta=\underbrace{\int_0^{h  } \cos( (h -s)\sqrt{\mathcal{A}}) f\Big(\alpha^{\intercal}(s)W(t_n)\Big)ds}_{=:\textmd{  I}^{v}}\\
&+\underbrace{\int_0^{h  }\cos( (h -s)\sqrt{\mathcal{A}})f'\Big(\alpha^{\intercal}(s)W(t_n)\Big) \int_0^{s }  \frac{\sin( ( s-\theta)\sqrt{\mathcal{A}})}{\sqrt{\mathcal{A}}} f\big( u(t_n+\theta)\big)d\theta ds}_{=:\textmd{  II}^{v}}
\\
&+\underbrace{\int_0^{h  } \cos( (h -s)\sqrt{\mathcal{A}}) R_{f''}(t_n, s) \Big(\int_0^{s }  \frac{\sin( ( s-\theta)\sqrt{\mathcal{A}})}{\sqrt{\mathcal{A}}} f\big( u(t_n+\theta)\big)d\theta \Big)^2 ds}_{=:\textmd{  III}^{v}}.
\end{aligned}
\end{equation}

The main objective is to find some computable third-order approximations of I$^u$\&I$^v,$   II$^u$\&II$^v,$ and   III$^u$\&III$^v$, which will be derived one by one in the next subsection.

\subsubsection{Step 2. Transitional third-order scheme}
In this part, we proceed to construct a transitional third-order scheme based on the approximations of I$^u$\&I$^v,$   II$^u$\&II$^v,$ and   III$^u$\&III$^v$ derived blew.

 $\bullet$ \textbf{Approximation to   I$^u$ and   I$^v$.}

 We first deal with the integral I$^u$. To this end, consider the
 twisted variable  which  was  widely  used in the development
of low-regularity time discretizations   \cite{LI22,LI23,O18,O21,O22,R21,S21,Zhao23}.
In this paper, we  introduce a new twisted  function $$F(t_n+s):=\beta(-s) f\big( \alpha^{\intercal}(s) W(t_n)\big).$$
After splitting the term $$\frac{\sin( (h -s)\sqrt{\mathcal{A}}) }{\sqrt{\mathcal{A}}}=\alpha^{\intercal}(h ) \beta(-s),$$
I$^u$   can be expressed by the twisted  function $F(t_n+s):$
 \begin{equation*}
\begin{aligned}
  \textmd{  I}^{u}
  =&\int_0^{h  }  \alpha^{\intercal}(h ) \beta(-s) f\big( \alpha^{\intercal}(s) W(t_n)\big) ds=\int_0^{h  }  \alpha^{\intercal}(h ) F(t_n+s) ds.
\end{aligned}
\end{equation*}
Considering  the Newton–Leibniz
formula for the twisted  function $F$:
$$F(t_n+s)=F(t_n)+ \int_0^{s }F'(t_n+\zeta) d\zeta,$$
we get
 \begin{equation}\label{P1-new1}
\begin{aligned}
  \textmd{  I}^{u}=&\int_0^{h  }  \alpha^{\intercal}(h ) F(t_n) ds+\int_0^{h  }  \alpha^{\intercal}(h ) \int_0^{s }F'(t_n+\zeta) d\zeta ds\\
 =&\int_0^{h  }  \alpha^{\intercal}(h ) F(t_n) ds+\int_0^{h  }  \alpha^{\intercal}(h )
 F'(t_n+\zeta)(h -\zeta) d\zeta.
\end{aligned}
\end{equation}
According to the scheme of $\alpha(h )$ given in
\eqref{notations w},  direct calculation yields the following expression
$$
  \alpha^{\intercal}(h )=\alpha^{\intercal}(h -s) M(s)\  \textmd{with}\ M(s):=\left(
                                 \begin{array}{cc}
                                  \cos(s\sqrt{\mathcal{A}}) &  s \sinc( s\sqrt{\mathcal{A}}) \\
                                 - \sqrt{\mathcal{A}} \sin( s\sqrt{\mathcal{A}}) &  \cos(s\sqrt{\mathcal{A}})  \\
                                 \end{array}
                               \right)
$$
 for any $s\in\mathbb{R}.$
Then
\eqref{P1-new1} can be expressed as
 \begin{equation*}
\begin{aligned}
\textmd{  I}^{u}=&\int_0^{h  }  \alpha^{\intercal}(h ) F(t_n) ds+\int_0^{h  }  (h -s) \alpha^{\intercal}(h -2s)
 M(2s)F'(t_n+s) ds\\
 =&\int_0^{h  }  \alpha^{\intercal}(h ) F(t_n) ds+\int_0^{h  }  (h -s) \alpha^{\intercal}(h -2s)
 M(0)F'(t_n) ds\\
 &+\int_0^{h  }  (h -s) \alpha^{\intercal}(h -2s)\int_0^{s}
 \frac{d  \big(M(2\zeta)F'(t_n+\zeta)\big)}{d\zeta}  d\zeta ds.
\end{aligned}
\end{equation*}
For the term $F'(t_n+\zeta)$ appeared above, it can be computed as:
 \begin{equation}\label{DF}
\begin{aligned}
&F'(t_n+\zeta)\\=&
 \frac{d}{d\zeta} \big( \beta(-\zeta)\big) f\big( \alpha^{\intercal}(\zeta) W(t_n)\big)+
            \beta(-\zeta)\frac{d}{d\zeta}  f\big( \alpha^{\intercal}(\zeta) W(t_n)  \\
            =&  \left(
              \begin{array}{c}
                - \cos(\zeta\sqrt{\mathcal{A}}) f\big( \alpha^{\intercal}(\zeta) W(t_n)\big) -\zeta \sinc(\zeta\sqrt{\mathcal{A}})\frac{d}{d\zeta}  f\big( \alpha^{\intercal}(\zeta) W(t_n)\big)   \\
             -\sqrt{\mathcal{A}}\sin(\zeta\sqrt{\mathcal{A}})f\big( \alpha^{\intercal}(\zeta) W(t_n)\big)+ \cos(\zeta\sqrt{\mathcal{A}}) \frac{d}{d\zeta}  f\big( \alpha^{\intercal}(\zeta) W(t_n)\big) \\
              \end{array}
            \right)\\
        =&M(-\zeta)\Big(  - f\big( \alpha^{\intercal}(\zeta) W(t_n)\big),
                                   \frac{d}{d\zeta}  f\big( \alpha^{\intercal}(\zeta) W(t_n)\big)
                              \Big)^{\intercal}.
\end{aligned}
\end{equation}
This result further leads to
 \begin{equation*}
\begin{aligned}
 &  \frac{d  }{d\zeta} M(2\zeta)F'(t_n+\zeta) =\frac{d  }{d\zeta} M(\zeta)\left(
                                \begin{array}{c}
                                  - f\big( \alpha^{\intercal}(\zeta) W(t_n)\big) \\
                                   \frac{d}{d\zeta}  f\big( \alpha^{\intercal}(\zeta) W(t_n)\big)  \\
                                \end{array}
                              \right)
                               \\=&\left(
                                 \begin{array}{cc}
                                 -\sqrt{\mathcal{A}} \sin(\zeta\sqrt{\mathcal{A}}) &  \cos( \zeta\sqrt{\mathcal{A}}) \\
                                 - \mathcal{A} \cos( \zeta\sqrt{\mathcal{A}}) &  -\sqrt{\mathcal{A}} \sin(\zeta\sqrt{\mathcal{A}}) \\
                                 \end{array}
                               \right)\left(
                                \begin{array}{c}
                                  - f\big( \alpha^{\intercal}(\zeta) W(t_n)\big) \\
                                   \frac{d}{d\zeta}  f\big( \alpha^{\intercal}(\zeta) W(t_n)\big)  \\
                                \end{array}
                              \right)\\&\ \ +M(\zeta)  \left(
                                \begin{array}{c}
                                  - \frac{d  }{d\zeta}f\big( \alpha^{\intercal}(\zeta) W(t_n)\big) \\
                                   \frac{d^2}{d\zeta^2}  f\big( \alpha^{\intercal}(\zeta) W(t_n)\big)  \\
                                \end{array}
                              \right)
\\
                              =&\left(
                                \begin{array}{c}
                                 \sqrt{\mathcal{A}} \sin(\zeta\sqrt{\mathcal{A}})f\big( \alpha^{\intercal}(\zeta) W(t_n)\big)
                                 +\zeta \sinc( \zeta\sqrt{\mathcal{A}}) \frac{d^2}{d\zeta^2}  f\big( \alpha^{\intercal}(\zeta) W(t_n)\big)   \\
                                  \mathcal{A} \cos(\zeta\sqrt{\mathcal{A}}) f\big( \alpha^{\intercal}(\zeta) W(t_n)\big) +\cos(\zeta\sqrt{\mathcal{A}}) \frac{d^2}{d\zeta^2}  f\big( \alpha^{\intercal}(\zeta) W(t_n)\big)   \\
                                \end{array}
                              \right)
                             \\=&\beta(\zeta)\Upsilon(t_n,\zeta),
\end{aligned}
\end{equation*}
with the notation \begin{equation}\label{UPS}
\begin{aligned}\Upsilon(t_n,\zeta):=&   \mathcal{A} f\big( \alpha^{\intercal}(\zeta) W(t_n)\big)
                                 +\frac{d^2}{d\zeta^2}  f\big( \alpha^{\intercal}(\zeta) W(t_n)\big) \\
                                 =&   \mathcal{A} f\big( \alpha^{\intercal}(\zeta) W(t_n)\big)
                                 +   f''\big( \alpha^{\intercal}(\zeta) W(t_n)\big)\big(\gamma^{\intercal}(-\zeta) W(t_n),\gamma^{\intercal}(-\zeta) W(t_n) \big)\\&- f'\big( \alpha^{\intercal}(\zeta) W(t_n)\big)\mathcal{A}\alpha^{\intercal}(\zeta) W(t_n)\\
                                     =&f''\big( \alpha^{\intercal}(\zeta) W(t_n)\big)\big(\gamma^{\intercal}(-\zeta) W(t_n),\gamma^{\intercal}(-\zeta) W(t_n) \big)\\&-  f''\big( \alpha^{\intercal}(\zeta) W(t_n)\big)\big(\nabla \alpha^{\intercal}(\zeta) W(t_n),\nabla \alpha^{\intercal}(\zeta) W(t_n) \big)\\
                                     &+\rho  f\big( \alpha^{\intercal}(\zeta) W(t_n)\big)-\rho  f'\big( \alpha^{\intercal}(\zeta) W(t_n)\big)\alpha^{\intercal}(\zeta) W(t_n). \end{aligned}
\end{equation}
Here in the last equation, we use the definition of  $\mathcal{A}$  \eqref{notations A} and split it into $-\Delta $ and  $\rho$.
It is noted that the result of $\Upsilon(t_n,\zeta)$ is the key point  for reducing the regularity since it changes $ \mathcal{A}$ into $\nabla $ in the expression. This good aspect comes from the careful treatment of the twisted  function $F$ and the splitting of $\alpha^{\intercal}(h )$.

With the above results and the splitting of $\Upsilon(t_n,\zeta)$, we  derive that
\begin{equation*}
\begin{aligned}
\textmd{  I}^{u}=&\int_0^{h  }  \alpha^{\intercal}(h )  ds F(t_n) +\int_0^{h  }  (h -s) \alpha^{\intercal}(h -2s)
 M(0)ds F'(t_n) \\
 &+\int_0^{h  }  (h -s) \alpha^{\intercal}(h -2s)\int_0^{s}
\beta(\zeta)\Upsilon(t_n,\zeta)  d\zeta ds\\
 =&h ^2\sinc(  h \sqrt{\mathcal{A}}) f\big( u(t_{n})\big)  + \left(
                                                                   \begin{array}{c}
                                                                      h ^2/2\sinc(  h \sqrt{\mathcal{A}})  \\
                                                                       \frac{ \sin (  h \sqrt{\mathcal{A}}) -  h \sqrt{\mathcal{A}}\cos(  h \sqrt{\mathcal{A}})}{2 \sqrt{\mathcal{A}}^3}\\
                                                                   \end{array}
                                                                 \right)^{\intercal}
 \left(
                                \begin{array}{c}
                                  - f\big( u(t_n)\big) \\
                                  f'\big( u(t_n)\big) v(t_n)  \\
                                \end{array}
                              \right) \\
 &+\int_0^{h  }  (h -s) \alpha^{\intercal}(h -2s)\int_0^{s}
\beta(\zeta)(\Upsilon(t_n,0)+\Upsilon(t_n,\zeta)-\Upsilon(t_n,0))  d\zeta ds\\
                              =&h ^2/2\sinc(  h \sqrt{\mathcal{A}}) f\big( u(t_{n})\big)+\frac{ \sin (  h \sqrt{\mathcal{A}}) -  h \sqrt{\mathcal{A}}\cos(  h \sqrt{\mathcal{A}})}{2 \sqrt{\mathcal{A}}^3} f'\big( u(t_n)\big) v(t_n) \\
 &+\frac{ 2-2\cos( h \sqrt{\mathcal{A}})-h \sqrt{\mathcal{A}}\sin(h \sqrt{\mathcal{A}})}{2   \mathcal{A}^2} \Upsilon(t_n,0)+R_1(t_n),\end{aligned}
\end{equation*}
where  the remainder $R_1$ is defined as
\begin{equation}\label{R1}R_1(t_n)=\int_0^{h  }  (h -s) \alpha^{\intercal}(h -2s)\int_0^{s}
\beta(\zeta)(\Upsilon(t_n,\zeta)-\Upsilon(t_n,0))  d\zeta ds,\end{equation}
and its boundedness will be shown in Lemma \ref{lamma1} of next section.
Based on the expression   \eqref{UPS}  of $\Upsilon$, it is trivial to get
 $$\Upsilon(t_n,0)=f''(u_n)\big(v(t_n)^2-(\nabla u(t_n))^2\big)+\rho  f(u(t_n))-\rho  f'(u(t_n))u(t_n).$$

In a very similar way, one deduces that \begin{equation}\label{  Iv bound2}
\begin{aligned}
  \textmd{  I}^{v}=&\int_0^{h  } \gamma^{\intercal}(-h ) \beta(-s) f\big( \alpha^{\intercal}(s) W(t_n)\big) ds\\
 =&\int_0^{h  } \gamma^{\intercal}(-h ) F(t_n) ds+\int_0^{h  }  \gamma^{\intercal}(-h )
 F'(t_n+\zeta)(h -\zeta) d\zeta\\
 =&\int_0^{h  }  \gamma^{\intercal}(-h ) ds F(t_n) +\int_0^{h  }  (h -s)  \gamma^{\intercal}(2s-h)
 M(0)ds F'(t_n) \\
 &+\int_0^{h  }  (h -s) \gamma^{\intercal}(2s-h)\int_0^{s}
\beta(\zeta)\Upsilon(t_n,\zeta)  d\zeta ds\\
                              =&h ( \cos(h \sqrt{\mathcal{A}}) + \sinc( h \sqrt{\mathcal{A}}) )/2 f\big( u(t_{n})\big)+h ^2/2\sinc(  h \sqrt{\mathcal{A}}) f'\big( u(t_n)\big) v(t_n) \\
 &+h^3\frac{ \sinc( h \sqrt{\mathcal{A}})-\cos(h \sqrt{\mathcal{A}})}{2(h \sqrt{\mathcal{A}})^2} \Upsilon(t_n,0)+R_2(t_n)
\end{aligned}
\end{equation}
with the remainder \begin{equation}\label{R2}R_2(t_n)=\int_0^{h  }  (h -s)  \gamma^{\intercal}(2s-h)\int_0^{s}
\beta(\zeta)(\Upsilon(t_n,\zeta)-\Upsilon(t_n,0))  d\zeta ds.\end{equation}
The estimate of this remainder will be given in  Lemma \ref{lamma2}  of next section.

$\bullet$ \textbf{Approximation to   II$^{u}$ and   II$^{v}$.}

Now we turn to   II$^{u}$ and   II$^{v}$.
For the first one, third-order local error  will be derived as follows.
In this paper, it is assumed that  the  nonlinear function $f$ of  \eqref{klein-gordon} satisfies   $\abs{f^{(k)}(w)}\leq C_0$  for $w\in \mathbb{R}$ and $k=1,2,3$. Moreover, we consider   the  regularity condition  $$(u(0,x), \partial_tu(0,x))\in H^{1+ \mu}(\mathbb{T}^d) \times H^{ \mu}(\mathbb{T}^d)$$
with $	
				\mu=
				\begin{cases}
					1, & d=1,\\
					d/2+\epsilon, & d=2,3,
				\end{cases}$ and $\epsilon>0$.
From
Proposition \ref{lamma00} and Lemma
\ref{lamma0}, it follows that for $d=1$
   \begin{equation}\label{  II bound}
\begin{aligned}
&\norm{\textmd{  II}^{u}}_{H^{1}}\\
\lesssim& \norm{ \int_0^{h  } \frac{ \sin( (h -s)\sqrt{\mathcal{A}})}{\sqrt{\mathcal{A}}} f'\Big(\alpha^{\intercal}(s)W(t_n)\Big) \int_0^{s }  \frac{\sin( ( s-\theta)\sqrt{\mathcal{A}})}{\sqrt{\mathcal{A}}} f\big( u(t_n+\theta)\big)d\theta ds} _{H^{1}}\\
\lesssim &  \int_0^{h  }\abs{h -s}
 \int_0^{  s }\norm{ (  s-\theta)\sinc( (s-\theta)\sqrt{\mathcal{A}}) f'\Big(\alpha^{\intercal}(s)W(t_n)\Big) f\big( u(t_n+\theta)\big)} _{H^{1}}d\theta ds\\
 \lesssim &  h ^4  \max_{s\in[0 ,h ]} \norm{f'\Big(\alpha^{\intercal}(s)W(t_n)\Big)} _{H^{1}} \max_{\zeta\in[0 ,h ]} \norm{f\big( u(t_n+\zeta)\big)} _{H^{1}} \lesssim   h ^4.
\end{aligned}
\end{equation}
 When $d=2$ or $3$, the  Kato–Ponce inequality \eqref{KPI} leads to
$$\norm{\textmd{  II}^{u}}_{H^{1}}
\lesssim  h ^4  \max_{s\in[0 ,h ]} \norm{f'\Big(\alpha^{\intercal}(s)W(t_n)\Big)} _{H^{1}} \max_{\zeta\in[0 ,h ]} \norm{f\big( u(t_n+\zeta)\big)} _{H^{\mu}} \lesssim   h ^4.
$$
However, for  the part $\textmd{  II}^{v}$, we only get the following estimate
   \begin{equation}\label{  IIv app}
\begin{aligned}
 \textmd{  II}^{v}=&\int_0^{h  }\cos( (h -s)\sqrt{\mathcal{A}})f'\Big(\alpha^{\intercal}(s)W(t_n)\Big) \int_0^{s }  \frac{\sin( ( s-\theta)\sqrt{\mathcal{A}})}{\sqrt{\mathcal{A}}} f\big( u(t_n+\theta)\big)d\theta ds\\
=& \int_0^{h  }\cos( (h -s)\sqrt{\mathcal{A}})f'(u(t_n)) \int_0^{s }  \frac{\sin( ( s-\theta)\sqrt{\mathcal{A}})}{\sqrt{\mathcal{A}}}d\theta ds  f\big( u(t_n)\big)+R_3(t_n)\\
=& h^3\frac{ \sinc( h \sqrt{\mathcal{A}})-\cos(h \sqrt{\mathcal{A}})}{2(h \sqrt{\mathcal{A}})^2} f'(u(t_n))  f\big( u(t_n)\big)+R_3(t_n),
\end{aligned}
\end{equation}
where $R_3$ is a remainder defined by
   \begin{equation}\label{R3}
\begin{aligned}
&R_3(t_n)\\
=&\int_0^{h  }\cos( (h -s)\sqrt{\mathcal{A}})f'\Big(\alpha^{\intercal}(s)W(t_n)\Big) \int_0^{s }  \frac{\sin( ( s-\theta)\sqrt{\mathcal{A}})}{\sqrt{\mathcal{A}}} f\big( u(t_n+\theta)\big)d\theta ds\\
&-\int_0^{h  }\cos( (h -s)\sqrt{\mathcal{A}})f'(u(t_n)) \int_0^{s }  \frac{\sin( ( s-\theta)\sqrt{\mathcal{A}})}{\sqrt{\mathcal{A}}}d\theta ds  f\big( u(t_n)\big).
\end{aligned}
\end{equation}
The bound of $R_3(t_n)$  will also be studied in Lemma \ref{lamma2}  of next section.

 $\bullet$ \textbf{Approximation to   III$^{u}$ and   III$^{v}$.}

 Finally, we pay attention to the bounds of   III$^{u}$ and   III$^{v}$ which are respectively presented in \eqref{local errors 1-new} and \eqref{local errorsv 1-new}.
Using Proposition \ref{lamma00} again, the following
estimate holds
   \begin{equation*}
\begin{aligned}
  &\norm{\textmd{  III}^{u}}_{H^{1}}\\\lesssim &\norm{ \int_0^{h  }  \frac{\sin( (h -s)\sqrt{\mathcal{A}})}{\sqrt{\mathcal{A}}} R_{f''}(t_n, s) \Big(\int_0^{s }  \frac{\sin( ( s-\theta)\sqrt{\mathcal{A}})}{\sqrt{\mathcal{A}}} f\big( u(t_n+\theta)\big)d\theta \Big)^2 ds} _{H^{1}}\\
\lesssim &  \int_0^{h  } (h -s)
 \norm{R_{f''}(t_n, s) \Big(\int_0^{s }( s-\theta)  \sinc( ( s-\theta)\sqrt{\mathcal{A}})  f\big( u(t_n+\theta)\big) d\theta \Big)^2} _{H^{1}}ds.
\end{aligned}
\end{equation*}
In view of \eqref{RF2}, we obtain for $d=1$
   \begin{equation}\label{  III bound}
\begin{aligned}
   \norm{\textmd{  III}^{u}}_{H^{1}}    \lesssim&   \int_0^{h} (h -s)
 \norm{R_{f''}(t_n, s)} _{H^{1}} \int_0^{s }( s-\theta)^2    \norm{ f\big( u(t_n+\theta)\big)}^2 _{H^{1}} d\theta  ds\\
 \lesssim  & h ^5    \max_{s\in[0 ,h ]} \norm{R_{f''}(t_n, s)} _{H^{1}}  \max_{\zeta\in[0 ,h ]} \norm{f\big( u(t_n+\zeta)\big)} ^2_{H^{1}}
 \lesssim    h ^5
\end{aligned}
\end{equation}
and  for $d=2$ or $3$
$$   \norm{\textmd{  III}^{u}}_{H^{1}}    \lesssim h ^5    \max_{s\in[0 ,h ]} \norm{R_{f''}(t_n, s)} _{H^{1}}  \max_{\zeta\in[0 ,h ]} \norm{f\big( u(t_n+\zeta)\big)} ^2_{H^{\mu}}
 \lesssim    h ^5.$$
By the same arguments, it is arrived that
   \begin{equation}\label{  IIIv bound}
\begin{aligned}
  &\norm{\textmd{  III}^{v}}_{L^{2}}\lesssim\norm{\textmd{  III}^{v}}_{H^{1}} \\\lesssim &\norm{ \int_0^{h  }   \cos( (h -s)\sqrt{\mathcal{A}}) R_{f''}(t_n, s) \Big(\int_0^{s }  \frac{\sin( ( s-\theta)\sqrt{\mathcal{A}})}{\sqrt{\mathcal{A}}} f\big( u(t_n+\theta)\big)d\theta \Big)^2 ds} _{L^{2}}\\
\lesssim &  \int_0^{h  }
 \norm{R_{f''}(t_n, s) \Big(\int_0^{s }( s-\theta)  \sinc( ( s-\theta)\sqrt{\mathcal{A}})  f\big( u(t_n+\theta)\big) d\theta \Big)^2} _{H^{1}}ds\\
  \lesssim  & h ^4    \max_{s\in[0 ,h ]} \norm{R_{f''}(t_n, s)} _{H^{1}}  \max_{\zeta\in[0 ,h ]} \norm{f\big( u(t_n+\zeta)\big)} ^2_{H^{ \mu}}
 \lesssim    h ^4.
\end{aligned}
\end{equation}

As a result, these two parts can be dropped in the numerical scheme without bringing any impact on the accuracy and regularity requirement.

Based on these results, we define a transitional method for  the Klein--Gordon
equation \eqref{klein-gordon} (Definition \ref{def-se}) by considering \eqref{Duhamel} as well as \eqref{local errors 1-new}-\eqref{local errorsv 1-new} and dropping the remainders $R_1(t_n),R_2(t_n),R_3(t_n)$ and   II$^{u}$,   III$^{u}$,   III$^{v}$ appeared in the above formulation:
 \begin{equation}\label{method-3}
\begin{aligned}
u_{n+1}=&\cos(h \sqrt{\mathcal{A}}) u_n+h  \sinc( h \sqrt{\mathcal{A}})  v_n+h^2\Phi_1(h \sqrt{\mathcal{A}}) f( u_n)\\&+h^3  \Psi_1(h \sqrt{\mathcal{A}})f'( u_n)v_n +h^4  \Psi_2(h \sqrt{\mathcal{A}}) F_1(u_n,v_n),\\
v_{n+1}
   =&-h \mathcal{A} \sinc( h \sqrt{\mathcal{A}}) u_n+\cos(h \sqrt{\mathcal{A}}) v_n+
h  \Phi_2(h \sqrt{\mathcal{A}})  f( u_n) \\&+h^2\Phi_1(h \sqrt{\mathcal{A}})f'( u_n)v_n +h^3  \Psi_1(h \sqrt{\mathcal{A}}) \big(F_1(u_n,v_n)+f'(u_n) f(u_n) \big),
\end{aligned}
\end{equation}
\subsubsection{Step 3. Symmetrization procedure}
In this part, we present the final construction of the fourth-order scheme, which is derived from the transitional scheme and a symmetrization procedure.
First, we take one step forward with $n=1$ in the transitional scheme to obtain $x_1$ and $v_1$. For $n \ge 2$, the values of $x_n$ and $v_n$ are determined by a two-step symmetric trigonometric integrator. The detailed construction is outlined as follows.
We denote the transitional scheme \eqref{method-3} by
\begin{equation}\label{method-4}
\begin{aligned}
\Psi_{h}(u_n,v_n):=&u_{n+1}-\cos(h \sqrt{\mathcal{A}}) u_n-h \sinc( h \sqrt{\mathcal{A}}) v_n\\
=&h^2\Phi_1(h \sqrt{\mathcal{A}}) f( u_n)+h^3 \Psi_1(h \sqrt{\mathcal{A}})f'( u_n)v_n +h^4 \Psi_2(h \sqrt{\mathcal{A}}) F_1(u_n,v_n),\\
\Phi_{h}(u_n,v_n):=&v_{n+1}+h \mathcal{A} \sinc( h \sqrt{\mathcal{A}}) u_n-\cos(h \sqrt{\mathcal{A}}) v_n\\
=&h \Phi_2(h \sqrt{\mathcal{A}}) f( u_n) +h^2\Phi_1(h \sqrt{\mathcal{A}})f'( u_n)v_n\\& +h^3 \Psi_1(h \sqrt{\mathcal{A}}) \big(F_1(u_n,v_n)+f'(u_n) f(u_n) \big).
\end{aligned}
\end{equation} 
Subsequently, we consider the following two-step symmetric trigonometric integrator:
\begin{equation}\label{tsei}
\begin{aligned}
u_{n+1}=&\cos(2h \sqrt{\mathcal{A}}) u_{n-1}+h \sinc(2 h \sqrt{\mathcal{A}}) v_{n-1}+\Psi_{h}(u_n,v_n)\\
&-\cos(2h \sqrt{\mathcal{A}}) \Psi_{-h}(u_n,v_n)+h \sinc(2 h \sqrt{\mathcal{A}}) \Phi_{-h}(u_n,v_n),\\
v_{n+1}=&-h \mathcal{A} \sinc(2 h \sqrt{\mathcal{A}}) u_{n-1}+\cos(2h \sqrt{\mathcal{A}}) v_{n-1}+\Phi_{h}(u_n,v_n)\\
&-h \mathcal{A} \sinc(2 h \sqrt{\mathcal{A}})\Psi_{-h}(u_n,v_n)-\cos(2h \sqrt{\mathcal{A}}) \Phi_{-h}(u_n,v_n).
\end{aligned}
\end{equation}
Substituting the transitional scheme \eqref{method-4} into this formulation and simplifying the resulting expressions yields the final scheme \eqref{method}. This completes the construction.
  
\section{Convergence}\label{sec:3}
In this section, we shall derive the convergence  of the proposed   integrator.
We will first present the main result and then  prove the error estimates in a  low regularity condition.
\subsection{Main result}

\begin{mytheo}  \label{symplectic thm}
Let the  nonlinear function $f$ of the Klein--Gordon equation \eqref{klein-gordon} satisfy   the Lipschitz continuity conditions $\abs{f^{(k)}(w)}\leq C_0$  for $w\in \mathbb{R}$ and $k=1,2,3,4$.
Under the  regularity condition  \begin{equation}\label{rc exact}(u(0,x), \partial_tu(0,x))\in [H^{2+ \mu }(\mathbb{T}^d)\bigcap H^{1+ \mu }_0(\mathbb{T}^d)]\times H^{ \mu}(\mathbb{T}^d)\end{equation}
with $	
				\mu=
				\begin{cases}
					1, & d=1,\\
					d/2+\epsilon, & d=2,3,
				\end{cases}$ and $\epsilon>0$,  there exist positive constants $C$ and $h_0$ such
that for any $h\in(0,h_0]$
 the numerical result $u_n,v_n$ produced  in Definition \ref{def-se} has the
global error:
 \begin{equation}\label{error bound}
\begin{aligned}
\max_{0\leq n \leq T/h } \norm{u_n-u(t_n)}_{H^{1}} \leq Ch ^4,\ \ \
\max_{0\leq n \leq T/h } \norm{v_n-v(t_n)}_{L^{2}}\leq Ch ^4,
\end{aligned}
\end{equation}
where $C$  depends only on $C_0$ and $T$.
\end{mytheo}
\begin{remark}
It is noted that for the traditional fourth-order integrators,  the estimates \eqref{error bound} usually hold under the
the requirement  $$(u(0,x), \partial_tu(0,x))\in  H^{5}(\mathbb{T}^d) \times H^{4}(\mathbb{T}^d).$$
In comparison with this condition, the regularity \eqref{rc exact} is lower.\end{remark}

\begin{remark}
Under the conditions of Theorem \ref{symplectic thm} and the regularity assumption
$$(u(0,x), \partial_tu(0,x))\in \left[H^{2+\mu}(\mathbb{T}^d)\cap H^{1+\mu}_0(\mathbb{T}^d)\right]\times H^{1+\mu}(\mathbb{T}^d),$$
the numerical solution generated by the fully-discrete scheme \eqref{method-f} satisfies the following error estimate:
\begin{equation}\label{error bound new}
\begin{aligned}
\norm{\Pi_{N_x}u(t_{n+1})-U_{n+1}}_{H^{1}}+ \norm{\Pi_{N_x}v(t_{n+1})-V_{n+1}}_{L^{2}}  \leq C\big( h ^4+ N_x^{-1- \mu}\big),
\end{aligned}
\end{equation}
where $0\leq n \leq \frac{T}{h}-1$, and $C$ denotes the error constant depending only on $T$ and $C_0$ from Theorem \ref{symplectic thm}.
This result follows from combining the proof of Theorem \ref{symplectic thm} with the standard error analysis of spectral methods. For brevity, we omit the detailed proof. Consequently, by letting $N_x \to \infty$, the convergence of the semi-discretization established in Theorem \ref{symplectic thm} is recovered. In practical computations, one may select a sufficiently large $N_x$ such that the dominant error of the fully-discrete scheme originates from the temporal discretization.

\end{remark}

The proof of Theorem  \ref{symplectic thm} is given in the rest parts of this section.  

\subsection{Bounds on the remainders $R_1(t_n),R_2(t_n),R_3(t_n)$}
We begin with  the bounds on the remainders $R_1(t_n)$, $R_2(t_n)$ and   $R_3(t_n)$ which are presented in  \eqref{R1}, \eqref{R2} and \eqref{R3}, respectively. 
\begin{mylemma}\label{lamma1}
Under the  conditions of Theorem   \ref{symplectic thm},
the remainder $R_1(t_n)$ given in \eqref{R1}  is bounded by
   \begin{equation}\label{R1 bound result}
\begin{aligned}\norm{R_1(t_n)}_{H^{1}}\lesssim h^4.
\end{aligned}
\end{equation}
\end{mylemma}
\begin{proof}
By using the expression \eqref{R1},  $R_1(t_n)$ can be estimated as
   \begin{equation}\label{R1RR}
\begin{aligned}
&\norm{R_1(t_n)}_{H^{1}} \\ =  &  \norm{\int_0^{h  } \int_0^{s} (h -s)  \frac{ \sin( (h-2s+\zeta)\sqrt{\mathcal{A}}) }{\sqrt{\mathcal{A}}}\big( \Upsilon(t_n,\zeta)-\Upsilon(t_n,0)\big)d\zeta ds}_{H^{1}} \\
\lesssim&   \int_0^{h  } (h -s) \int_0^{s} \norm{    \sin( (h-2s+\zeta)\sqrt{\mathcal{A}})   \frac{\Upsilon(t_n,\zeta)-\Upsilon(t_n,0) }{\sqrt{\mathcal{A}}} }_{H^{1}}d\zeta ds \\
\lesssim&   \int_0^{h  } (h -s) \int_0^{s}  \norm{ \frac{\Upsilon(t_n,\zeta)-\Upsilon(t_n,0) }{\sqrt{\mathcal{A}}} }_{H^{1}}d\zeta ds\\
\lesssim&  h ^3  \max_{\zeta \in[0,h ]} \norm{\frac{ \Upsilon(t_n,\zeta)-\Upsilon(t_n,0)}{\sqrt{\mathcal{A}}}}_{H^{1}}
\lesssim   h ^3  \max_{\zeta \in[0,h ]} \norm{\Upsilon(t_n,\zeta)-\Upsilon(t_n,0)}_{L^{2}}.
\end{aligned}
\end{equation}
Then  we are devoted to the bound of $\Upsilon(t_n,\zeta)-\Upsilon(t_n,0)$. According to the result \eqref{UPS} of  $\Upsilon$, we decompose  $\Upsilon(t_n,\zeta)-\Upsilon(t_n,0)$ into four parts
   \begin{equation}\label{mUpsilon}
\begin{aligned}
&\Upsilon(t_n,\zeta)-\Upsilon(t_n,0)\\
=&\underbrace{f''\big( \alpha^{\intercal}(\zeta) W(t_n)\big)\big(\gamma^{\intercal}(-\zeta) W(t_n) \big)^2-
f''\big( \alpha^{\intercal}(0) W(t_n)\big)\big(\gamma^{\intercal}(0) W(t_n) \big)^2}_{=:\Delta_1 f}\\&
+ \underbrace{f''\big( \alpha^{\intercal}(0) W(t_n)\big)\big(\nabla \alpha^{\intercal}(0) W(t_n) \big)^2-  f''\big( \alpha^{\intercal}(\zeta) W(t_n)\big)\big(\nabla \alpha^{\intercal}(\zeta) W(t_n) \big)^2}_{=:\Delta_2 f}\\
                                     &+\underbrace{\rho  f\big( \alpha^{\intercal}(\zeta) W(t_n)\big)-\rho  f\big( \alpha^{\intercal}(0) W(t_n)\big)}_{=:\Delta_3 f}\\
                                     &+\underbrace{\rho  f'\big( \alpha^{\intercal}(0) W(t_n)\big)\alpha^{\intercal}(0) W(t_n)-\rho  f'\big( \alpha^{\intercal}(\zeta) W(t_n)\big)\alpha^{\intercal}(\zeta) W(t_n)}_{=:\Delta_4 f}.
                                     \end{aligned}
\end{equation}
In what follows, we  will deduce the results for these four terms one by one.

 $\bullet$ \textbf{Bound on $\norm{\Delta_1 f}_{L^{2}}$.}

We first recall the Kato–Ponce inequalities given in Lemma \ref{lamma0} and as a result
   \begin{equation*}
\begin{aligned}
\norm{\Delta_1 f}_{L^{2}}
\lesssim&\norm{f''\big( \alpha^{\intercal}(\zeta) W(t_n)\big)-f''\big( \alpha^{\intercal}(0) W(t_n)\big)}_{H^{\mu}}\norm{\big(\gamma^{\intercal}(-\zeta) W(t_n) \big)^2}_{L^{2}}\\
&+\norm{f''\big( \alpha^{\intercal}(0) W(t_n)\big)}_{H^{\mu}}\norm{\big(\gamma^{\intercal}(-\zeta) W(t_n) \big)^2-\big(\gamma^{\intercal}(0) W(t_n) \big)^2}_{L^{2}}\\
\lesssim&\norm{f''\big( \alpha^{\intercal}(\zeta) W(t_n)\big)-f''\big( \alpha^{\intercal}(0) W(t_n)\big)}_{H^{\mu}}\norm{\big(\gamma^{\intercal}(-\zeta) W(t_n) \big)^2}_{L^{2}}\\
&+\norm{f''\big( \alpha^{\intercal}(0) W(t_n)\big)}_{H^{\mu}}\norm{ \gamma^{\intercal}(-\zeta) W(t_n) +\gamma^{\intercal}(0) W(t_n) }_{H^{\mu}}\\&\ \ \ \norm{ \gamma^{\intercal}(-\zeta) W(t_n)  - \gamma^{\intercal}(0) W(t_n)  }_{L^{2}}.\end{aligned}
\end{equation*}
Based on this result, it is needed to estimate
   \begin{equation*}
\begin{aligned}
&\norm{f''\big( \alpha^{\intercal}(\zeta) W(t_n)\big)-f''\big( \alpha^{\intercal}(0) W(t_n)\big)}_{H^{\mu}}\\
\lesssim&\norm{  \alpha^{\intercal}(\zeta) W(t_n)- \alpha^{\intercal}(0) W(t_n)}_{H^{\mu}}\\
 \lesssim&\norm{  \cos(\zeta\sqrt{\mathcal{A}}) u(t_n)-u(t_n)+\zeta \sinc( \zeta\sqrt{\mathcal{A}})  v(t_n)}_{H^{\mu}}
 \\
 \lesssim&\norm{  -2\sin(\zeta\sqrt{\mathcal{A}}/2) \sinc(\zeta\sqrt{\mathcal{A}}/2) \zeta\sqrt{\mathcal{A}} u(t_n)}_{H^{\mu}}+\zeta\norm{ v(t_n)}_{H^{\mu}} \\
 \lesssim& \zeta\big(\norm{  u(t_n)}_{H^{1+\mu}}+\norm{ v(t_n)}_{H^{\mu}}\big)\end{aligned}
\end{equation*}
by using Proposition \ref{lamma00}.  Then the Sobolev embedding theorem shows that
   \begin{equation*}
\begin{aligned}
&\norm{ (\gamma^{\intercal}(-\zeta) W(t_n))^2}_{L^{2}}\\
 \lesssim&  \norm{ \gamma^{\intercal}(-\zeta) W(t_n)}^2_{L^{4}}
  \lesssim  \norm{ \gamma^{\intercal}(-\zeta) W(t_n)}^2_{W^{1,p}}  \quad (W^{1,p}\hookrightarrow L^{4})\\
   \lesssim& \norm{ \gamma^{\intercal}(-\zeta) W(t_n)}^2_{W^{\frac{d}{4},2}}\quad   (W^{\frac{d}{4},2}\hookrightarrow W^{1,p})\\
 \lesssim &\norm{ \gamma^{\intercal}(-\zeta) W(t_n)}^2_{H^{\frac{d}{4}}}
 \lesssim  \norm{-\sqrt{\mathcal{A}} \sin(\zeta\sqrt{\mathcal{A}}) u(t_n)+\cos(\zeta\sqrt{\mathcal{A}}) v(t_n)}^2_{H^{\frac{d}{4}}}\\
  \lesssim&  \norm{  u(t_n)}^2_{H^{1+\frac{d}{4}}}+\norm{ v(t_n)}^2_{H^{\frac{d}{4}}}+2\norm{  u(t_n)}_{H^{1+\frac{d}{4}}}\norm{ v(t_n)}_{H^{\frac{d}{4}}}.
\end{aligned}
\end{equation*}
Meanwhile, it follows from \eqref{notations w} that
   \begin{equation*}
\begin{aligned}
 &\norm{ \gamma^{\intercal}(-\zeta) W(t_n) +\gamma^{\intercal}(0) W(t_n) }_{H^{\mu}} \\\lesssim& \norm{-\sqrt{\mathcal{A}} \sin(\zeta\sqrt{\mathcal{A}}) u(t_n)+\cos(\zeta\sqrt{\mathcal{A}}) v(t_n)+v(t_n)} _{H^{\mu}}\\
 \lesssim& \norm{ u(t_n)} _{H^{1+\mu}}+\norm{ v(t_n)} _{H^{\mu}}.
\end{aligned}
\end{equation*}
Analogously, we obtain
   \begin{equation*}
\begin{aligned}
 &\norm{ \gamma^{\intercal}(-\zeta) W(t_n) -\gamma^{\intercal}(0) W(t_n) }_{L^{2}}\\ \lesssim& \norm{-\sqrt{\mathcal{A}} \sin(\zeta\sqrt{\mathcal{A}}) u(t_n)+\cos(\zeta\sqrt{\mathcal{A}}) v(t_n)-v(t_n)} _{L^{2}}\\
\lesssim& \norm{\zeta\sqrt{\mathcal{A}} \sqrt{\mathcal{A}} \sinc(\zeta\sqrt{\mathcal{A}}) u(t_n)} _{L^{2}}+\norm{ -2\sin(\zeta\sqrt{\mathcal{A}}/2) \sinc(\zeta\sqrt{\mathcal{A}}/2) \zeta\sqrt{\mathcal{A}}  v(t_n)} _{L^{2}}\\
  \lesssim& \zeta \norm{ u(t_n)} _{H^{2}}+ \zeta \norm{   v(t_n)} _{H^{1}}.
\end{aligned}
\end{equation*}
Under the regularity condition
$$(u(0,x), \partial_tu(0,x))\in  H^{2}(\mathbb{T}^d) \times H^{1}(\mathbb{T}^d),$$
one gets
$$\norm{ \gamma^{\intercal}(-\zeta) W(t_n) -\gamma^{\intercal}(0) W(t_n) }_{L^{2}}   \lesssim  \zeta.$$

 Overall, collecting all the estimates together yields
   \begin{equation*}
\begin{aligned}
\norm{\Delta_1 f}_{L^{2}}
\lesssim& \zeta\big(\norm{  u(t_n)}_{H^{1+\mu}}+\norm{ v(t_n)}_{H^{\mu}}\big) \big(\norm{  u(t_n)}^2_{H^{1+\frac{d}{4}}}+\norm{ v(t_n)}^2_{H^{\frac{d}{4}}}\\&+2\norm{  u(t_n)}_{H^{1+\frac{d}{4}}}\norm{ v(t_n)}_{H^{\frac{d}{4}}} + \norm{ u(t_n)} _{H^{2}}+  \norm{   v(t_n)} _{H^{1}}\big)
\lesssim\zeta.\end{aligned}
\end{equation*}

 $\bullet$ \textbf{Bound on $\norm{\Delta_2 f}_{L^{2}}$.}

By the similar argument as $\Delta_1 f$,  we decompose $\Delta_2 f$ into four terms:
   \begin{equation*}
\begin{aligned}
&\norm{\Delta_2 f}_{L^{2}}\\
=&\norm{f''\big( \alpha^{\intercal}(\zeta) W(t_n)\big)\big(\nabla \alpha^{\intercal}(\zeta) W(t_n) \big)^2-
f''\big( \alpha^{\intercal}(0) W(t_n)\big)\big(\nabla \alpha^{\intercal}(0) W(t_n) \big)^2}_{L^{2}}\\
\lesssim&\norm{f''\big( \alpha^{\intercal}(\zeta) W(t_n)\big)-f''\big( \alpha^{\intercal}(0) W(t_n)\big)}_{H^{\mu}}\norm{\big(\nabla \alpha^{\intercal}(\zeta) W(t_n) \big)^2}_{L^{2}}\\
&+\norm{f''\big( \alpha^{\intercal}(0) W(t_n)\big)}_{H^{\mu}}\norm{\big(\nabla \alpha^{\intercal}(\zeta) W(t_n) \big)^2-\big(\nabla \alpha^{\intercal}(0) W(t_n) \big)^2}_{L^{2}}.\end{aligned}
\end{equation*}
It is noted that this formula contains the same expression as   $\Delta_1 f$ and thus we only need to estimate the different two terms: $$\norm{\big(\nabla \alpha^{\intercal}(\zeta) W(t_n) \big)^2}_{L^{2}}$$ and $$\norm{\big(\nabla \alpha^{\intercal}(\zeta) W(t_n) \big)^2-\big(\nabla \alpha^{\intercal}(0) W(t_n) \big)^2}_{L^{2}}.$$ For the first one, it follows from  the Sobolev embedding theorem
that
    \begin{equation*}
\begin{aligned}
&\norm{ (\nabla \alpha^{\intercal}(\zeta) W(t_n))^2}_{L^{2}}\\
 \lesssim&  \norm{\nabla \alpha^{\intercal}(\zeta) W(t_n)}^2_{L^{4}}
  \lesssim  \norm{ \nabla \alpha^{\intercal}(\zeta) W(t_n)}^2_{W^{1,p}}  \quad (W^{1,p}\hookrightarrow L^{4})\\
   \lesssim& \norm{ \nabla \alpha^{\intercal}(\zeta) W(t_n)}^2_{W^{\frac{d}{4},2}}\quad   (W^{\frac{d}{4},2}\hookrightarrow W^{1,p})\\
 \lesssim &\norm{ \nabla \alpha^{\intercal}(\zeta) W(t_n)}^2_{H^{\frac{d}{4}}}
 \lesssim  \norm{\cos(\zeta\sqrt{\mathcal{A}}) u(t_n)+  \sin( \zeta\sqrt{\mathcal{A}})\frac{v(t_n)}{\sqrt{\mathcal{A}}}}^2_{H^{1+\frac{d}{4}}}\\
  \lesssim&  \norm{  u(t_n)}^2_{H^{1+\frac{d}{4}}}+\norm{ v(t_n)}^2_{H^{\frac{d}{4}}}+2\norm{  u(t_n)}_{H^{1+\frac{d}{4}}}\norm{ v(t_n)}_{H^{\frac{d}{4}}}.
\end{aligned}
\end{equation*}
 For the second one, we use the Kato–Ponce inequality to get
     \begin{equation*}
\begin{aligned}
&\norm{\big(\nabla \alpha^{\intercal}(\zeta) W(t_n) \big)^2-\big(\nabla \alpha^{\intercal}(0) W(t_n) \big)^2}_{L^{2}}  \\
\lesssim& \norm{\nabla \alpha^{\intercal}(\zeta) W(t_n) +\nabla \alpha^{\intercal}(0) W(t_n) }_{H^{\mu}}\norm{\nabla \alpha^{\intercal}(\zeta) W(t_n) -\nabla \alpha^{\intercal}(0) W(t_n) }_{L^{2}}\\
\lesssim &\norm{\cos(\zeta\sqrt{\mathcal{A}}) u(t_n)+  \sin( \zeta\sqrt{\mathcal{A}})\frac{v(t_n)}{\sqrt{\mathcal{A}}}+u(t_n)} _{H^{1+\mu}}\\& \norm{(\cos(\zeta\sqrt{\mathcal{A}})-1) u(t_n)+\zeta  \sinc( \zeta\sqrt{\mathcal{A}}) v(t_n) } _{H^{1}}\\
\lesssim &\big(\norm{ u(t_n)} _{H^{1+\mu}}+\norm{ v(t_n)} _{H^{\mu}}\big)\\& \big(\norm{ -2\sin(\zeta\sqrt{\mathcal{A}}/2) \sinc(\zeta\sqrt{\mathcal{A}}/2) \zeta\sqrt{\mathcal{A}}  u(t_n)} _{H^{1}}+\norm{\zeta\sinc(\zeta\sqrt{\mathcal{A}}) v(t_n)} _{H^{1}}\big)\\
\lesssim & \zeta \big(\norm{ u(t_n)} _{H^{1+\mu}}+\norm{ v(t_n)} _{H^{\mu}}\big) \big(\norm{ u(t_n)} _{H^{2}}+ \zeta \norm{   v(t_n)} _{H^{1}}\big).
\end{aligned}
\end{equation*}
 To summarize, we have
$$
\norm{\Delta_2 f}_{L^{2}} \lesssim \zeta.$$

  $\bullet$ \textbf{Bound on $\norm{\Delta_3 f}_{L^{2}}$.}

The following  estimate can be proved in the same way as above
     \begin{equation*}
\begin{aligned}
 &\norm{\Delta_3 f}_{L^{2}}\\
 =& \norm{ \rho  f\big( \alpha^{\intercal}(\zeta) W(t_n)\big)-\rho  f\big( \alpha^{\intercal}(0) W(t_n)\big)}_{L^{2}}\\
   \lesssim& \norm{  f'\Big( \big(\alpha^{\intercal}(\zeta)  -\varsigma(\alpha^{\intercal}(0) -\alpha^{\intercal}(\zeta) )\big)W(t_n)\Big)
    (\alpha^{\intercal}(0) -\alpha^{\intercal}(\zeta) )W(t_n)  }_{L^{2}}\\
       \lesssim& \norm{  f'\Big( \big(\alpha^{\intercal}(\zeta)  -\varsigma(\alpha^{\intercal}(0) -\alpha^{\intercal}(\zeta) )\big)W(t_n)\Big)}_{H^{\mu}} \norm{
    (\alpha^{\intercal}(0) -\alpha^{\intercal}(\zeta) )W(t_n)}_{L^{2}}\\
 \lesssim &\norm{  \alpha^{\intercal}(\zeta) W(t_n)- \alpha^{\intercal}(0) W(t_n)}_{L^{2}}
  \lesssim \zeta\big(\norm{  u(t_n)}_{H^{1}}+\norm{ v(t_n)}_{L^{2}}\big),
  \end{aligned}
\end{equation*}
where $\varsigma\in[0,1]$.

   $\bullet$ \textbf{Bound on $\norm{\Delta_4 f}_{L^{2}}$.}

   For the last part, we analogously obtain
        \begin{equation*}
\begin{aligned}
 \norm{\Delta_4 f}_{L^{2}}=& \norm{ \rho  f'\big( \alpha^{\intercal}(0) W(t_n)\big)\alpha^{\intercal}(0) W(t_n)-\rho  f'\big( \alpha^{\intercal}(\zeta) W(t_n)\big)\alpha^{\intercal}(\zeta) W(t_n)}_{L^{2}}\\
\lesssim&\norm{f'\big( \alpha^{\intercal}(\zeta) W(t_n)\big)-f'\big( \alpha^{\intercal}(0) W(t_n)\big)}_{H^{\mu}}\norm{ \alpha^{\intercal}(\zeta) W(t_n)}_{L^{2}}\\
&+\norm{f'\big( \alpha^{\intercal}(0) W(t_n)\big)}_{H^{\mu}}\norm{\alpha^{\intercal}(\zeta) W(t_n) - \alpha^{\intercal}(0) W(t_n)}_{L^{2}}\\
\lesssim&\zeta\big(\norm{  u(t_n)}_{H^{1+\mu}}+\norm{ v(t_n)}_{H^{\mu}}\big)\big(\norm{  u(t_n)}_{H^{1}}+\zeta\norm{ v(t_n)}_{L^{2}}\big).
  \end{aligned}
\end{equation*}
 Consequently, based on the above results, we arrive at
 \begin{equation}\label{zeta- bound result}\max_{\zeta \in[0,h ]} \norm{\Upsilon(t_n,\zeta)-\Upsilon(t_n,0)}_{L^{2}}\lesssim\max_{\zeta \in[0,h ]}   \zeta \lesssim h.\end{equation}
 Combining this with \eqref{R1RR} immediately gives the statement \eqref{R1 bound result}  and  the proof of this lemma is complete.
 \hfill $\blacksquare$
\end{proof}

\begin{mylemma}\label{lamma2}
Suppose that the  conditions of Theorem   \ref{symplectic thm} hold. For
the remainders $R_2(t_n)$ and $R_3(t_n)$   respectively defined in \eqref{R2} and \eqref{R3}, they are bounded by
   \begin{equation}\label{R2 bound result}
\begin{aligned}\norm{R_2(t_n)}_{L^{2}}\lesssim h^4,\ \ \ \norm{R_3(t_n)}_{L^{2}}\lesssim h^4.
\end{aligned}
\end{equation}
\end{mylemma}
\begin{proof}
Keeping \eqref{zeta- bound result} in mind, the following  estimate can be proved in the same way  as Lemma \ref{lamma1}
\begin{equation*}
\begin{aligned}
\norm{R_2(t_n)}_{L^{2}}= &\norm{\int_0^{h  } \int_0^{s} (h -s)  \cos( (h-2s+\zeta)\sqrt{\mathcal{A}})  \big( \Upsilon(t_n,\zeta)-\Upsilon(t_n,0)\big)d\zeta ds}_{L^{2}} \\
\lesssim& \int_0^{h  } \int_0^{s} \abs{h -s} \norm{ \big( \Upsilon(t_n,\zeta)-\Upsilon(t_n,0)\big)}_{L^{2}}d\zeta ds \\
\lesssim&  h ^3  \max_{\zeta \in[0,h ]} \norm{ \Upsilon(t_n,\zeta)-\Upsilon(t_n,0) }_{L^{2}}\lesssim h^4.
\end{aligned}
\end{equation*}
Then we  represent the scheme of $R_3(t_n)$ as
   \begin{equation*}
\begin{aligned}
&R_3(t_n)\\
=&\int_0^{h  }\cos( (h -s)\sqrt{\mathcal{A}})f'\big(\alpha^{\intercal}(s)W(t_n)\big) \int_0^{s }  \frac{\sin( ( s-\theta)\sqrt{\mathcal{A}})}{\sqrt{\mathcal{A}}} f\big( u(t_n+\theta)\big)d\theta ds\\
&-\int_0^{h  }\cos( (h -s)\sqrt{\mathcal{A}})f'(u(t_n)) \int_0^{s }  \frac{\sin( ( s-\theta)\sqrt{\mathcal{A}})}{\sqrt{\mathcal{A}}}d\theta ds  f\big( u(t_n)\big)
\end{aligned}
\end{equation*}
and split it into
   \begin{equation}\label{R3n}
\begin{aligned}
R_3(t_n)
=&\int_0^{h  }\cos( (h -s)\sqrt{\mathcal{A}})\big(f'\big(\alpha^{\intercal}(s)W(t_n)\big)-f'(u(t_n))\big)\\& \int_0^{s }  \frac{\sin( ( s-\theta)\sqrt{\mathcal{A}})}{\sqrt{\mathcal{A}}} f\big( u(t_n+\theta)\big)d\theta ds+\int_0^{h  }\cos( (h -s)\sqrt{\mathcal{A}})f'(u(t_n)) \\
&\int_0^{s }  \frac{\sin( ( s-\theta)\sqrt{\mathcal{A}})}{\sqrt{\mathcal{A}}} \big(f\big( u(t_n+\theta)\big)-f\big( u(t_n)\big)\big)d\theta ds.
\end{aligned}
\end{equation}
Some derivations lead to
    \begin{equation*}
\begin{aligned}
&\norm{R_3(t_n)}_{L^{2}}\\
\lesssim &  \int_0^{h  }\int_0^{ s } \norm{ f'\big(\alpha^{\intercal}(s)W(t_n)\big)-f'(u(t_n))} _{H^{\mu}}  \abs{s-\theta}\norm{ f\big( u(t_n+\theta)\big)} _{L^{2}}d\theta ds \\
&+  \int_0^{h  }\int_0^{ s } \norm{ f'(u(t_n))} _{H^{\mu}}  \abs{s-\theta}\norm{ f\big( u(t_n+\theta)\big)-f\big( u(t_n)\big)} _{L^{2}}d\theta ds \\
\lesssim &  h ^3  \max_{s\in[0,h]} \norm{ f'\big(\alpha^{\intercal}(s)W(t_n)\big)-f'(u(t_n))} _{H^{\mu}}\\
&+h ^3  \max_{\theta\in[0,h]} \norm{f\big( u(t_n+\theta)\big)-f\big( u(t_n)\big)} _{L^{2}}\\
\lesssim &  h ^3  \max_{s\in[0,h]} \norm{ f''\big(\varsigma_1(\alpha^{\intercal}(s)W(t_n)-u(t_n))+u(t_n) \big)\big(\alpha^{\intercal}(s)W(t_n)-u(t_n)\big)} _{H^{\mu}}\\
&+h ^3  \max_{\theta\in[0,h]} \norm{ f'\big(\varsigma_2( u(t_n+\theta)-u(t_n))+u(t_n) \big)\big( u(t_n+\theta)-u(t_n)\big)}_{L^{2}}\\
\lesssim &  h ^3  \max_{s\in[0,h]} \norm{ f''\big(\varsigma_1(\alpha^{\intercal}(s)W(t_n)-u(t_n))+u(t_n) \big)} _{H^{\mu}}\norm{\alpha^{\intercal}(s)W(t_n)-u(t_n)} _{H^{\mu}}\\
&+h ^3  \max_{\theta\in[0,h]} \norm{ f'\big(\varsigma_2( u(t_n+\theta)-u(t_n))+u(t_n) \big)} _{H^{\mu}}\norm{ u(t_n+\theta)-u(t_n)}_{L^{2}}\\
\lesssim &  h ^3  \max_{s\in[0,h]} \norm{  \alpha^{\intercal}(s)W(t_n) -u(t_n)} _{H^{\mu}}+h ^3  \max_{\theta\in[0,h]} \norm{ u(t_n+\theta) -  u(t_n)} _{L^{2}}\\
 \lesssim&  h ^3  \max_{s\in[0,h]} s\big(\norm{  u(t_n)}_{H^{1+\mu}}+\norm{ v(t_n)}_{H^{\mu}}\big) + h ^3  \max_{\theta\in[0,h]} \theta  \norm{  v(t_n)}_{L^{2}}
  \lesssim  h ^4
\end{aligned}
\end{equation*}
with $\varsigma_1, \varsigma_2\in[0,1]$.
The proof is complete.  \hfill $\blacksquare$
\end{proof}

\subsection{Proof of   Theorem \ref{symplectic thm}}
Sofar we have derived the bounds for the remainders which are dropped in the numerical scheme.
Based on them, we can present the error analysis (the proof of   Theorem \ref{symplectic thm}) in what follows.
\begin{proof}
\textbf{(I) Error estimates of the transitional scheme \eqref{method-3}.}
%
%

To estimate the error of the transitional scheme \eqref{method-3}
$$e_n^{u}:=u(t_{n})-\Psi_{h}(u_{n-1},v_{n-1}),\quad e_n^{v}:=v(t_{n})-\Phi_{h}(u_{n-1},v_{n-1}),\quad    0\leq n\leq  T /h ,$$
we shall first consider the local truncation errors which are defined by inserting
the solution of  \eqref{klein-gordon} into \eqref{method-3}:
 \begin{equation}\label{local errors}
\begin{aligned}
\zeta_n^{u}:=&u(t_{n+1})-\cos(h \sqrt{\mathcal{A}}) u(t_n)-h  \sinc( h \sqrt{\mathcal{A}})  v(t_n)-h^2\Phi_1(h \sqrt{\mathcal{A}}) f( u(t_n))\\&-h^3  \Psi_1(h \sqrt{\mathcal{A}})f'( u(t_n))v(t_n) -h^4  \Psi_2(h \sqrt{\mathcal{A}})  \Big(f''(u(t_n))v(t_n)^2\\&-  f''(u(t_n))(\nabla u(t_n))^2+\rho  f(u(t_n))-\rho  f'(u(t_n))u(t_n)\Big),\\
\zeta_n^{v}:=&v(t_{n+1})+h \mathcal{A} \sinc( h \sqrt{\mathcal{A}}) u(t_n)-\cos(h \sqrt{\mathcal{A}}) v(t_n)-
h  \Phi_2(h \sqrt{\mathcal{A}})  f( u(t_n)) \\&-h^2\Phi_1(h \sqrt{\mathcal{A}})f'( u(t_n))v(t_n) -h^3  \Psi_1(h \sqrt{\mathcal{A}})  \Big(f''(u(t_n))(v(t_n)^2\\&-  (\nabla u(t_n))^2)+\rho  f(u(t_n))- f'(u(t_n))(\rho u(t_n)- f(u(t_n)))\Big),
\end{aligned}
\end{equation}
where $n=0,1,\ldots, T/h-1$.  For the coefficient functions appeared above, it follows from  \eqref{method-coefficient} that the  functions $\Phi_1(m),\Phi_2(m),m\Psi_1(m),m^2\Psi_2(m)$ are uniformly bounded for any
$m\in \mathbb{R}$ which immediately leads to
\begin{equation}\label{method-coefficient-bound}
\begin{aligned}
 &\norm{\Phi_1(h \sqrt{\mathcal{A}}) y}_{H^{\nu}}\lesssim \norm{y}_{H^{\nu}},\qquad \ \ \  \norm{\Phi_2(h \sqrt{\mathcal{A}}) y}_{H^{\nu}}\lesssim \norm{y}_{H^{\nu}},\\
   &\norm{h \sqrt{\mathcal{A}} \Psi_1(h \sqrt{\mathcal{A}}) y}_{H^{\nu}}\lesssim \norm{y}_{H^{\nu}},\ \ \norm{(h \sqrt{\mathcal{A}})^2\Psi_2(h \sqrt{\mathcal{A}}) y}_{H^{\nu}}\lesssim \norm{y}_{H^{\nu}},
\end{aligned}
\end{equation}
where we assume that $y\in H^{\nu}$ with any $\nu\geq 0$.

Subtracting the corresponding local error terms \eqref{local errors} from the scheme  \eqref{method-3}, we get the recurrence relation
for the errors
 \begin{equation}\label{equ errors}
\begin{aligned}
&e_{n+1}^{u}- \cos(h \sqrt{\mathcal{A}})e_n^{u}-h  \sinc( h \sqrt{\mathcal{A}})e_{n}^{v}=\zeta_n^{u}+h^2\eta_n^{u},\\&
e_{n+1}^{v}+ h \mathcal{A} \sinc( h \sqrt{\mathcal{A}})e_n^{u}-\cos(h \sqrt{\mathcal{A}})e_n^{v}=\zeta_n^{v}+h\eta_n^{v},\ n=0,1,\ldots, T/h-1,
\end{aligned}
\end{equation}
where   we denote
 \begin{equation}\label{sta errors}
\begin{aligned}
\eta_n^{u}:=&\Phi_1(h \sqrt{\mathcal{A}}) \big(f( u_n)-f( u(t_n))\big) +h   \Psi_1(h \sqrt{\mathcal{A}})\big(f'( u_n)v_n-f'( u(t_n))v(t_n)\big)\\& +h^2  \Psi_2(h \sqrt{\mathcal{A}})  \Big(\big(f''(u_n)v_n^2-f''(u(t_n))v(t_n)^2\big) \\&-  \big(f''(u_n)(\sqrt{\mathcal{A}}u_n)^2-f''(u(t_n))(\nabla u(t_n))^2\big)\\
&+\rho  \big(f(u_n)-f(u(t_n))\big)-\rho  \big(f'(u_n)u_n-f'(u(t_n))u(t_n)\big)\Big),\\
 \eta_n^{v}:=& \Phi_2(h \sqrt{\mathcal{A}})   \big(f( u_n)-f( u(t_n))\big) +h\Phi_1(h \sqrt{\mathcal{A}})\big(f'( u_n)v_n-f'( u(t_n))v(t_n)\big)\\&  +h^2  \Psi_1(h \sqrt{\mathcal{A}})  \Big(\big(f''(u_n)v_n^2-f''(u(t_n))v(t_n)^2\big) \\&-  \big(f''(u_n)(\sqrt{\mathcal{A}}u_n)^2-f''(u(t_n))(\nabla u(t_n))^2\big)\\
&+\rho  \big(f(u_n)-f(u(t_n))\big)-\rho  \big(f'(u_n)u_n-f'(u(t_n))u(t_n)\big) \\&-  \big( f'(u_n) f(u_n)- f'(u(t_n)) f(u(t_n))\big)\Big).
\end{aligned}
\end{equation}
The starting value of \eqref{equ errors} is  $e_0^{u}=e_0^{v}=0.$

In what follows, the proof is  divided into  three parts. The first one is about   the boundedness of numerical solution,
the second is devoted to local errors $\zeta_n^{u}, \zeta_n^{v}$ and stability $\eta_n^{u}, \eta_n^{v}$, and the last one concerns global errors $e_n^{u}, e_n^{v}$.

$\bullet$ In this part, we prove the boundedness of numerical solution:
\begin{equation}\label{bound num}
 \norm{u_{n} } _{H^{1+ \mu}}+\norm{v_{n}}_{H^{ \mu}}\lesssim 1,\ \ \ n=0,1,\ldots, T/h.
\end{equation}
For $n=0$,   (\ref{bound num})  is obviously true. Then we assume  that it holds  up to some $n=1,2,\ldots,m$, and we shall show that (\ref{bound num}) is true for $m+1$.
To this end, we reformulate
the scheme   \eqref{method-3}
 as
\begin{equation*}
\begin{aligned}
&\left(
  \begin{array}{c}
u_{m+1}\\
v_{m+1} \\
  \end{array}
\right)=\left(
  \begin{array}{cc}
         \cos(h \sqrt{\mathcal{A}})       & h  \sinc( h \sqrt{\mathcal{A}}) \\
    - h \mathcal{A} \sinc( h \sqrt{\mathcal{A}})  &  \cos(h \sqrt{\mathcal{A}}) \\
  \end{array}
\right) \left(
  \begin{array}{c}
   u_{m} \\
     v_{m} \\
  \end{array}
\right)
\\&+h \left(
          \begin{array}{c}
            h \Phi_1  f( u_m) +h^2 \Psi_1 f'( u_m)v_m +h^3  \Psi_2  F_1(u_m,v_m)\\
            \Phi_2   f( u_m)  +h\Phi_1 f'( u_m)v_m +h^2  \Psi_1  \big(F_1(u_m,v_m)+f'(u_m) f(u_m) \big)\\
          \end{array}
        \right).
\end{aligned}
\end{equation*}
Using the notation $\norm{\left(
  \begin{array}{c}
  w_1\\
   w_2 \\
  \end{array}
\right)}_{1+\mu }:=\sqrt{\norm{ w_1}_{H^{1+\mu }}^2+\norm{w_2}_{H^{\mu }}^2}$, it is derived that
  $$\norm{\left(
  \begin{array}{cc}
         \cos(h \sqrt{\mathcal{A}})       & h  \sinc( h \sqrt{\mathcal{A}}) \\
    - h \mathcal{A} \sinc( h \sqrt{\mathcal{A}})  &  \cos(h \sqrt{\mathcal{A}}) \\
  \end{array}
\right) \left(
  \begin{array}{c}
    u_m \\
  v_m \\
  \end{array}
\right) }_{1+\mu }\leq\norm{ \left(
  \begin{array}{c}
    u_m \\
  v_m \\
  \end{array}
\right) }_{1+\mu }.$$
According to the boundedness   (\ref{bound num}) up to  $m$, the expression \eqref{method-coefficient} of the coefficient functions $ \Psi_1, \Psi_2, \Phi_1, \Phi_2$  and the assumption on $f$ given in Theorem  \ref{symplectic thm}, it is easy to see that
\begin{equation*}
\begin{aligned}
&\norm{\left(
  \begin{array}{c}
u_{m+1}\\
v_{m+1} \\
  \end{array}
\right)}_{1+\mu }\leq \norm{\left(
  \begin{array}{c}
u_{m}\\
v_{m} \\
  \end{array}
\right)}_{1+\mu }
+   Ch .
\end{aligned}
\end{equation*}
Applying the Gronwall inequality  immediately yields  (\ref{bound num}) for $n=m+1$.

$\bullet$ By the bounds \eqref{ II bound}, \eqref{ III bound}-\eqref{ IIIv bound} and Lemmas \ref{lamma1}--\ref{lamma2},
the local errors $\zeta_n^{u}, \zeta_n^{v}$ are estimated as
\begin{equation*}
\begin{aligned}&\norm{\zeta_n^{u}}_{H^1}\lesssim  \norm{R_1(t_n)}_{H^1} +\norm{\textmd{Part II}^{u}}_{H^{1}}+\norm{\textmd{Part III}^{u}}_{H^{1}} \lesssim  h ^4,\\ &\norm{\zeta_n^{v}}_{L^2}\lesssim  \norm{R_2(t_n)}_{L^2} +\norm{R_3(t_n)}_{L^2}+\norm{\textmd{Part III}^{v}}_{L^2} \lesssim  h ^4.\end{aligned}
\end{equation*}
 To establish the estimate for  the stability  terms $\eta_n^{u}$ and $\eta_n^{v}$ defined in \eqref{sta errors},
we  consider the   terms of \eqref{sta errors} separately. For the first and last terms, it is easy to see that
 \begin{equation*}
\begin{aligned}
&h^2  \norm{\Psi_2(h \sqrt{\mathcal{A}}) \Big(\rho  \big(f(u_n)-f(u(t_n))\big)-\rho  \big(f'(u_n)u_n-f'(u(t_n))u(t_n)\big)\Big)}_{H^1}\\
& +\norm{\Phi_1(h \sqrt{\mathcal{A}}) \big(f( u_n)-f( u(t_n))\big)}_{H^1}\lesssim \norm{e_n^{u}}_{H^1},
\end{aligned}
\end{equation*}
 Based on the  Lipschitz continuity conditions, Kato–Ponce inequalities and \eqref{method-coefficient-bound},
it is derived that
 \begin{equation*}
\begin{aligned}
&\norm{h   \Psi_1(h \sqrt{\mathcal{A}})\big(f'( u_n)v_n-f'( u(t_n))v(t_n)\big)}_{H^1}\\
=&
\norm{h \sqrt{\mathcal{A}}  \Psi_1(h \sqrt{\mathcal{A}}) \frac{f'( u_n)v_n-f'( u(t_n))v(t_n)}{\sqrt{\mathcal{A}}} }_{H^1}
\\\lesssim&\norm{ f'( u_n)v_n-f'( u(t_n))v(t_n) }_{L^2}\\
\lesssim&\norm{f'( u_n)}_{H^{\mu}}\norm{e_n^{v}}_{L^2}+
\norm{f'( u_n) -f'( u(t_n)) }_{L^2}\norm{v(t_n)}_{H^{\mu}}
\\\lesssim& \norm{e_n^{v}}_{L^2}+ \norm{e_n^{u}}_{L^2},
\end{aligned}
\end{equation*}
and
 \begin{equation*}
\begin{aligned}
&\norm{h^2  \Psi_2(h \sqrt{\mathcal{A}})  \big(f''(u_n)v_n^2-f''(u(t_n))v(t_n)^2\big)}_{H^1}\\
=&h
\norm{h \sqrt{\mathcal{A}}  \Psi_2(h \sqrt{\mathcal{A}}) \frac{f''(u_n)v_n^2-f''(u(t_n))v(t_n)^2}{\sqrt{\mathcal{A}}} }_{H^1}\\
\lesssim&h\norm{f''(u_n)v_n^2-f''(u(t_n))v(t_n)^2}_{L^2}
\lesssim h(\norm{e_n^{v}}_{L^2}+ \norm{e_n^{u}}_{L^2}).
\end{aligned}
\end{equation*}
With the same arguments, we get
 \begin{equation*}
\begin{aligned}
&\norm{h^2  \Psi_2(h \sqrt{\mathcal{A}})  \big(f''(u_n)(\nabla u_n)^2-f''(u(t_n))(\nabla u(t_n))^2\big)}_{H^1}\\
=&
\norm{h^2  \mathcal{A}   \Psi_2(h \sqrt{\mathcal{A}}) \frac{f''(u_n)(\nabla u_n)^2-f''(u(t_n))(\nabla u(t_n))^2}{ \mathcal{A}} }_{H^1}\\
\lesssim  &
\norm{ \frac{f''(u_n)(\nabla u_n)^2-f''(u(t_n))(\nabla u(t_n))^2}{ \mathcal{A}} }_{H^1}\\
\lesssim  &
\norm{ \frac{\big(f''(u_n)-f''(u(t_n))\big)(\nabla u_n)^2}{ \sqrt{\mathcal{A}}} }_{L^2} +\norm{ \frac{f''(u(t_n))\big((\nabla u_n)^2-(\nabla u(t_n))^2\big)}{\sqrt{ \mathcal{A}}} }_{L^2}.
\end{aligned}
\end{equation*}
According to the results of Lemma \ref{lamma0} and the boundedness \eqref{bound num}, we find
 \begin{equation*}
\begin{aligned}
&
\norm{ \frac{\big(f''(u_n)-f''(u(t_n))\big)(\nabla u_n)^2}{ \sqrt{\mathcal{A}}} }_{L^2}\\
\lesssim&   \norm{f''(u_n)-f''(u(t_n))}_{L^2}\norm{u_n\nabla u_n}_{H^{\mu}} \\
\lesssim &
\norm{u^2_n}_{H^{1+\mu}}\norm{f''(u_n)-f''(u(t_n))}_{L^2} \lesssim  \norm{e_n^{u}}_{H^1},
\end{aligned}
\end{equation*}
and
 \begin{equation*}
\begin{aligned}
&
\norm{ \frac{f''(u(t_n))\big((\nabla u_n)^2-(\nabla u(t_n))^2\big)}{\sqrt{ \mathcal{A}}} }_{L^2}\\
\lesssim &  \norm{f''(u(t_n)}_{H^{\mu}}\norm{\nabla u^2_n - \nabla u^2(t_n) }_{L^2}\\
\lesssim &
\norm{  u^2_n - u^2(t_n) }_{H^1}\norm{f''(u(t_n)}_{H^{\mu}}\\
\lesssim& \norm{  u_n + u(t_n) }_{H^{\mu }} \norm{e_n^{u}}_{H^1} \lesssim \norm{e_n^{u}}_{H^1}.
\end{aligned}
\end{equation*}
Here the boundedness \eqref{bound num} of numerical solution is used in the derivation.
Therefore, it is obtained that
\begin{equation*}
\begin{aligned}
&\norm{h^2  \Psi_2(h \sqrt{\mathcal{A}})  \big(f''(u_n)(\nabla u_n)^2-f''(u(t_n))(\nabla u(t_n))^2\big)}_{H^1} \lesssim  \norm{e_n^{u}}_{H^1}.
\end{aligned}
\end{equation*}

Combining these results with the first formula of \eqref{sta errors} gives
$$
\norm{\eta_n^{u}}_{H^1}
\lesssim  \norm{e_n^{u}}_{H^1}+\norm{e_n^{v}}_{L^2}.
$$
In a similar way,  we can get the same result for $\eta_n^{v}$:$$
\norm{\eta_n^{v}}_{L^2}
\lesssim  \norm{e_n^{u}}_{H^1}+\norm{e_n^{v}}_{L^2}.$$

$\bullet$
 Now we turn back to the recurrence relation \eqref{equ errors} which is rewritten in the form:
\begin{equation*}
\begin{aligned}
&\left(
  \begin{array}{c}
  e_{n+1}^{u}\\
   e_{n+1}^{v} \\
  \end{array}
\right)=\left(
  \begin{array}{cc}
         \cos(h \sqrt{\mathcal{A}})       & h  \sinc( h \sqrt{\mathcal{A}}) \\
    - h \mathcal{A} \sinc( h \sqrt{\mathcal{A}})  &  \cos(h \sqrt{\mathcal{A}}) \\
  \end{array}
\right) \left(
  \begin{array}{c}
    e_{n}^u \\
     e_{n}^v \\
  \end{array}
\right)
+ \left(
          \begin{array}{c}
            \zeta_n^{u}+h^2\eta_n^{u}\\
            \zeta_n^{v}+h\eta_n^{v}\\
          \end{array}
        \right).
\end{aligned}
\end{equation*}
It is bounded by
\begin{equation*}
\begin{aligned}
&\norm{\left(
  \begin{array}{c}
  e_{n+1}^{u}\\
   e_{n+1}^{v} \\
  \end{array}
\right)}_1\leq \norm{\left(
  \begin{array}{c}
    e_{n}^u \\
     e_{n}^v \\
  \end{array}
\right)}_1
+ \norm{\left(
          \begin{array}{c}
            \zeta_n^{u}+h^2\eta_n^{u}\\
            \zeta_n^{v}+h\eta_n^{v}\\
          \end{array}
        \right)}_1\\
        \leq& \norm{\left(
  \begin{array}{c}
    e_{n}^u \\
     e_{n}^v \\
  \end{array}
\right)}_1
+  \sqrt{(\norm{\zeta_n^{u}}_{H^1}+h^2\norm{\eta_n^{u}}_{H^1})^2+
   (\norm{\zeta_n^{v}}_{L^2}+h^2\norm{\eta_n^{v}}_{L^2})^2}\\
           \leq& \norm{\left(
  \begin{array}{c}
    e_{n}^u \\
     e_{n}^v \\
  \end{array}
\right)}_1
+  C\sqrt{h^8+h^6( \norm{e_n^{u}}_{H^1}+\norm{e_n^{v}}_{L^2})+h^4(\norm{e_n^{u}}^2_{H^1}+\norm{e_n^{v}}^2_{L^2})}\\
           \leq& \norm{\left(
  \begin{array}{c}
    e_{n}^u \\
     e_{n}^v \\
  \end{array}
\right)}_1
+  C h^4+Ch^3 \sqrt{\norm{e_n^{u}}_{H^1}+\norm{e_n^{v}}_{L^2}}+Ch^2 \sqrt{\norm{e_n^{u}}^2_{H^1}+\norm{e_n^{v}}^2_{L^2}}.
\end{aligned}
\end{equation*}
Applying the Gronwall inequality  yields
$$
\sqrt{\norm{e_{n+1}^{u}}^2_{H^1}+\norm{e_{n+1}^{v}}^2_{L^2}}\leq C h^3.
$$
This shows  \begin{equation}\label{error bound-3}
\begin{aligned}
\max_{0\leq n \leq T/h } \norm{u_n-u(t_n)}_{H^{1}} \leq Ch ^3,\ \ \
\max_{0\leq n \leq T/h } \norm{v_n-v(t_n)}_{L^{2}}\leq Ch ^3,
\end{aligned}
\end{equation}
for the transitional scheme \eqref{method-3}.

\textbf{(II) Error estimates of the final scheme \eqref{method}.} 

With the notations $\mathbf{w} = (u, v)^\intercal$ and $\mathbf{w}_n = (u_n, v_n)^\intercal$, we
denote the final scheme \eqref{method}  as $$\mathbf{w}_{n+1}=\mathbf{\Upsilon}_h(\mathbf{w}_n).$$
For this scheme,  based on \eqref{tsei}, it is helpful to write
 \begin{equation}\label{FS-3}
\mathbf{\Upsilon}_h(\mathbf{w}_n)=  e^{ih \mathcal{K}}\mathbf{w}_n + \mathbf{F}(\mathbf{w}_n)
\end{equation}
with $\mathcal{K}=\left(
                                 \begin{array}{cc}
                                   0 & I \\
                                   -\mathcal{A} & 0 \\
                                 \end{array}
                               \right)$  and
\begin{equation*} 
\begin{aligned}
\mathbf{F}(\mathbf{w}_n)=& \left(
                             \begin{array}{c}
                              \Psi_{h}(u_n,v_n)-\cos(2h \sqrt{\mathcal{A}}) \Psi_{-h}(u_n,v_n)+h \sinc(2 h \sqrt{\mathcal{A}}) \Phi_{-h}(u_n,v_n)\\
                               \Phi_{h}(u_n,v_n)-h \mathcal{A} \sinc(2 h \sqrt{\mathcal{A}})\Psi_{-h}(u_n,v_n)-\cos(2h \sqrt{\mathcal{A}}) \Phi_{-h}(u_n,v_n) \\
                             \end{array}
                           \right).
\end{aligned}
\end{equation*}
Then we have
\begin{align*}
\|\mathbf{\Upsilon}_h(\mathbf{x}) -\mathbf{\Upsilon}_h(\mathbf{y})\|_{H^\alpha}  &\leq \left\|  e^{ih \mathcal{K}} (\mathbf{x}-\mathbf{y})\right\|_{H^\alpha}  
 + \left\| \mathbf{F}(\mathbf{x})-\mathbf{F}(\mathbf{y})\right\|_{H^\alpha} \\
&\leq \|\mathbf{x} - \mathbf{y}\|_{H^\alpha} + h C\bigl(\|\mathbf{x}\|_{H^\alpha}, \|\mathbf{y}\|_{H^\alpha}\bigr) \|\mathbf{x}-\mathbf{y}\|_{H^\alpha},
\end{align*}
where in the final line we used the fact that $ e^{ih \mathcal{K}},\cos(2h \sqrt{\mathcal{A}}),\sinc(2 h \sqrt{\mathcal{A}})$ are isometries on $H^\alpha (\alpha\geq 0)$ and the stability property of $\Psi_{h},\Phi_{h}$. This clearly implies  the following stability of the final scheme \eqref{method}
\[
\|\mathbf{\Upsilon}_h(\mathbf{x}) -\mathbf{\Upsilon}_h(\mathbf{y})\|_{H^\alpha} \leq \Big(1 + hC\bigl(\|\mathbf{x}\|_{H^\alpha}, \|\mathbf{y}\|_{H^\alpha}\bigr) \Big) \|\mathbf{x}-\mathbf{y}\|_{H^\alpha}. \tag{3.2}
\]

Then we estimate the local error $ \mathbf{w}(t_{n+1})-\mathbf{\Upsilon}_h(\mathbf{w}(t_{n})) $ of the final scheme \eqref{method}.
According to the scheme  \eqref{FS-3} and the local errors of transitional scheme defined in \eqref{local errors}, it is obtained that
\begin{align*}
 \mathbf{w}(t_{n+1})-\mathbf{\Upsilon}_h(\mathbf{w}(t_{n})) 
&= \bigl( \zeta_n^{u},\zeta_n^{v} \bigr)^\intercal-e^{2ih \mathcal{K}} \bigl( \zeta_{-n}^{u},\zeta_{-n}^{v} \bigr)^\intercal,
\end{align*}
where the notations $\zeta_{-n}^{u}$ and $\zeta_{-n}^{v}$ denote the local errors 
  $\zeta_{-n}^{u} := u(t_{n+1} - h) - \Psi_{-h}(u(t_{n+1}),v(t_{n+1}))$  and   $\zeta_{-n}^{v} := v(t_{n+1} - h) - \Phi_{-h}(u(t_{n+1}),v(t_{n+1}))$, respectively.
Due to symmetry and the similar analysis  as transitional scheme, this error can be improved to be
    \[
    |||\bigl( \zeta_n^{u},\zeta_n^{v} \bigr)^\intercal-e^{2ih \mathcal{K}} \bigl( \zeta_{-n}^{u},\zeta_{-n}^{v} \bigr)^\intercal|||_{1} \leq  C h^{5}
    \]
    under the  regularity condition   \eqref{rc exact}.
 Here the norm $ |||\bigl(x,y \bigr)^\intercal|||_{1}$ denotes  $$ |||\bigl(x,y \bigr)^\intercal|||_{1}=\|x\|_{H^{1}}+\|y\|_{L^{2}}.$$

Combining the above stability and local error estimates,  the global convergence result \eqref{error bound}  of our final method \eqref{method} can be proved immediately by   the standard error arguments. Therefore,  the theorem is
confirmed.  \hfill $\blacksquare$
\end{proof}

\section{Numerical test}\label{sec:4}

In this section, we show the numerical performance of the proposed fourth-order low-regularity integrator (LRI) by
comparing it with the well-known exponential integrators (EIs) \cite{Hochbruck10}.
 We choose a fourth-order EI   from \cite{Hochbruck06} (denoted by EIO4) and an  exponential fitting  third-order EI   from \cite{wang12} (denoted by EIO3). The reference solution is obtained by using a very fine stepsize for the fourth-order low-regularity integrator.

 We shall
present the results with an one-dimensional example of (\ref{klein-gordon})
for simplicity: the nonlinear  Klein--Gordon equation \eqref{klein-gordon} with $d=1, \rho=0,\mathbb{T}=(-\pi,\pi), f(u)=\sin(u)$
(we  note that  the scheme proposed in  \cite{Zhao23} is not applicable to this nonlinear function).
  We choose the initial values $\psi_1(x)$ and $\psi_2(x)$ in the same way  as described in Section 5.1 of \cite{O18} and Section 4 of \cite{Zhao23}. More precisely, the  initial values $\psi_1(x)$ and $\psi_2(x)$
are in the space  $H^{\theta}(\mathbb{T})\times H^{\theta-1}(\mathbb{T})$.

We first  verify the spatial accuracy of the
our method where  $h=10^{-6}$ is chosen such
 that the temporal error is
negligible compared to the spatial error.  Figure  \ref{p0new} plots the error:
$err=\frac{\norm{u_n-u(t_n)}_{H^{1}}}{\norm{u(t_n)}_{H^{1}}}+\frac{\norm{v_n-v(t_n)}_{L^{2}}}{\norm{v(t_n)}_{L^{2}}}$ under different $N_x$ and $\theta$, where the reference solution is obtained again by the LRI with $N_x = 2^{16}$.
From the results, it is seen that for the large  $N_x$  the
error  brought by the  spatial discretization   can be neglected. Thence, we choose $N_x=2^{12}$ in the experiment.

In this test, we report the global errors
\begin{equation*}
  err = \frac{\|u_n-u(t_n)\|_{H^{1}}}{\|u(t_n)\|_{H^{1}}}
      + \frac{\|v_n-v(t_n)\|_{L^{2}}}{\|v(t_n)\|_{L^{2}}}
\end{equation*}
at time $T=1$. The results for different initial data
$(\psi_1,\psi_2)\in H^{\theta}(\mathbb{T})\times H^{\theta-1}(\mathbb{T})$
are presented in Figure~\ref{p1}. It can be observed that, for initial data
in $H^{\theta}(\mathbb{T})\times H^{\theta-1}(\mathbb{T})$ with $\theta=4$,
all the methods attain their theoretical convergence orders. However, for
rougher initial data with $\theta=3,2$, only the low-regularity integrator
(LRI) proposed in this paper retains the expected fourth-order convergence
in $H^{1}(\mathbb{T})\times L^{2}(\mathbb{T})$, whereas the exponential
integrators suffer from a clear order reduction. This indicates that the
new integrator LRI performs significantly better than the exponential
integrators when the solution regularity is low. Notably, even when the
regularity assumption required for fourth-order convergence is slightly
violated ($\theta=2$), no significant order reduction is observed for the
LRI. This phenomenon merits further investigation in future work.

 To investigate the practical gain from our proposed integrator, we study the efficiency of all the methods.
 Figure \ref{fig2}  displays the error at $T=1$ against the CPU time. It can be seen from the results that our proposed method LRI  can reach the same error level with remarkably less CPU time for the  problem with lower regularity, and this clearly demonstrates the more efficiency of LRI  than the classical  exponential integrators.

  For the numerical experiments in two-dimensional  or three-dimensional space, the method has similar performance and we skip it for brevity.
    \begin{figure}[t!]
\centering
\includegraphics[width=3.9cm,height=4.4cm]{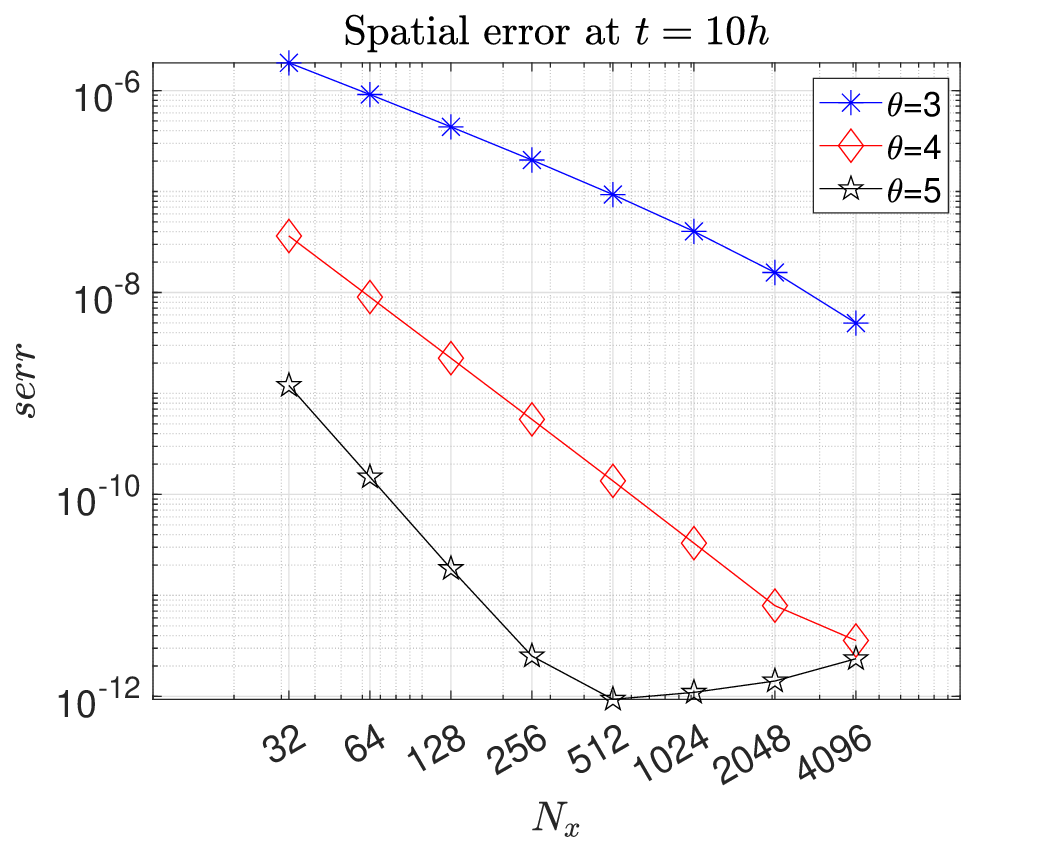}
\includegraphics[width=3.9cm,height=4.4cm]{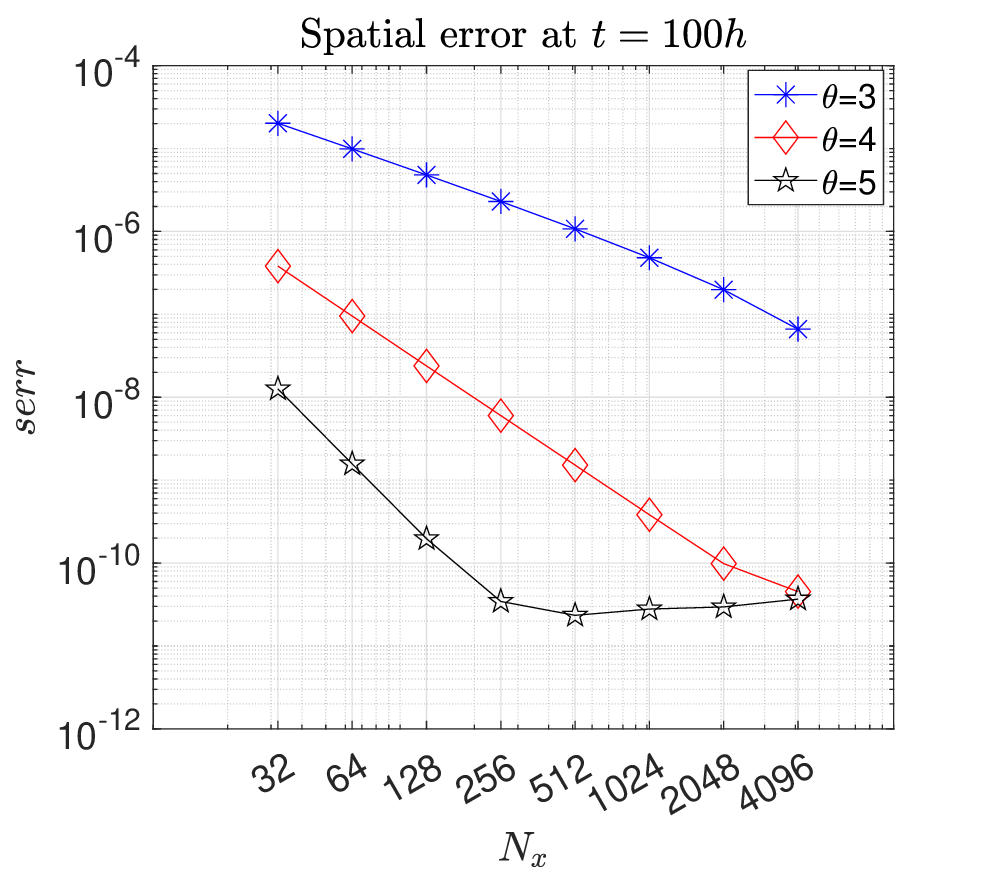}
\includegraphics[width=3.9cm,height=4.4cm]{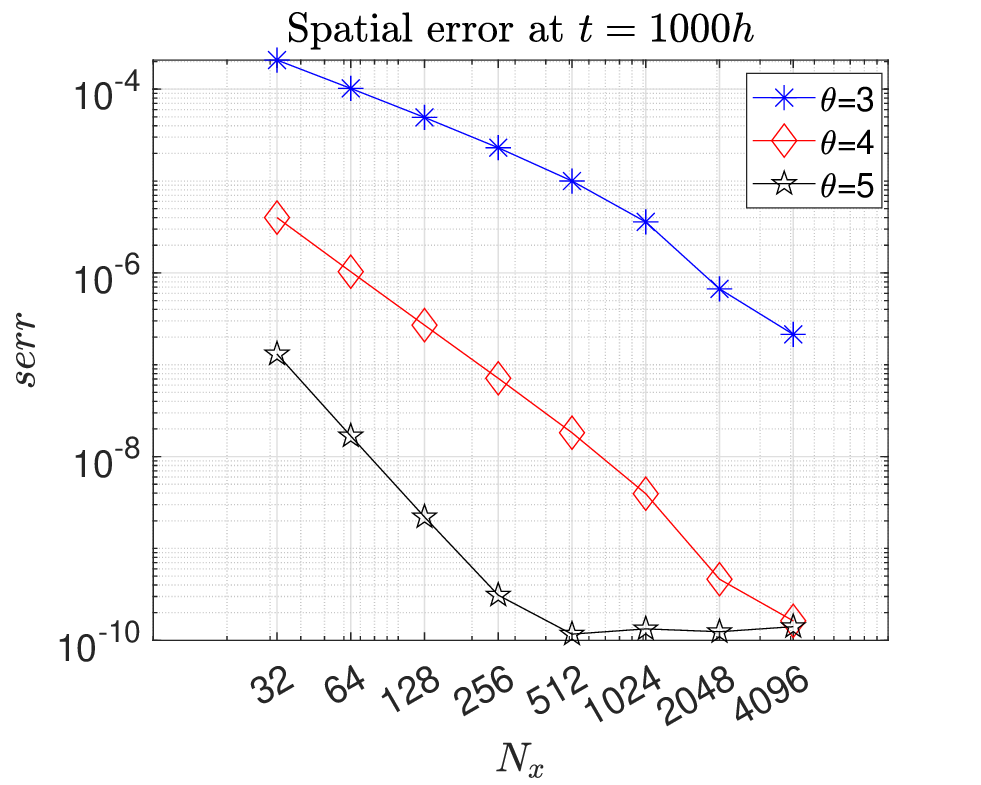}
\caption{Spatial error of LRI: the error $serr=\frac{\norm{u_n-u(t_n)}_{H^{1}}}{\norm{u(t_n)}_{H^{1}}}+\frac{\norm{v_n-v(t_n)}_{L^{2}}}{\norm{v(t_n)}_{L^{2}}}$ with  initial data $H^{\theta}(\mathbb{T}) \times H^{\theta-1}(\mathbb{T})$ against different $N_x$.} \label{p0new}
\end{figure}

    \begin{figure}[t!]
\centering
\includegraphics[width=3.9cm,height=4.4cm]{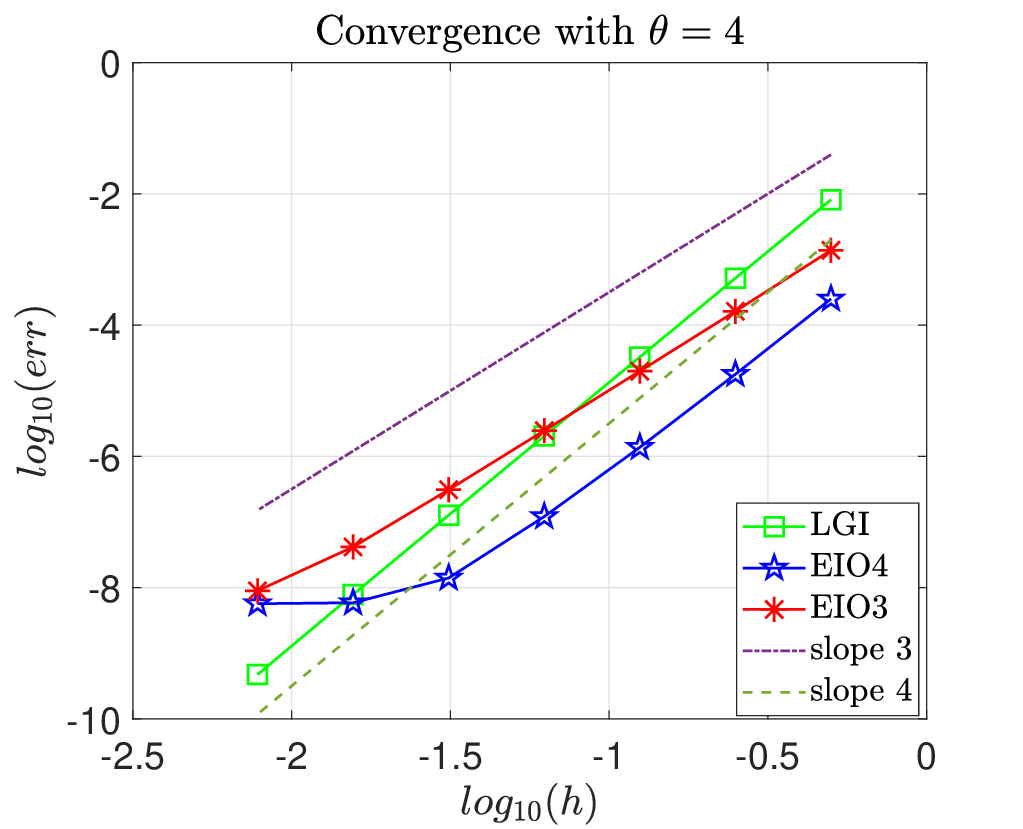}
\includegraphics[width=3.9cm,height=4.4cm]{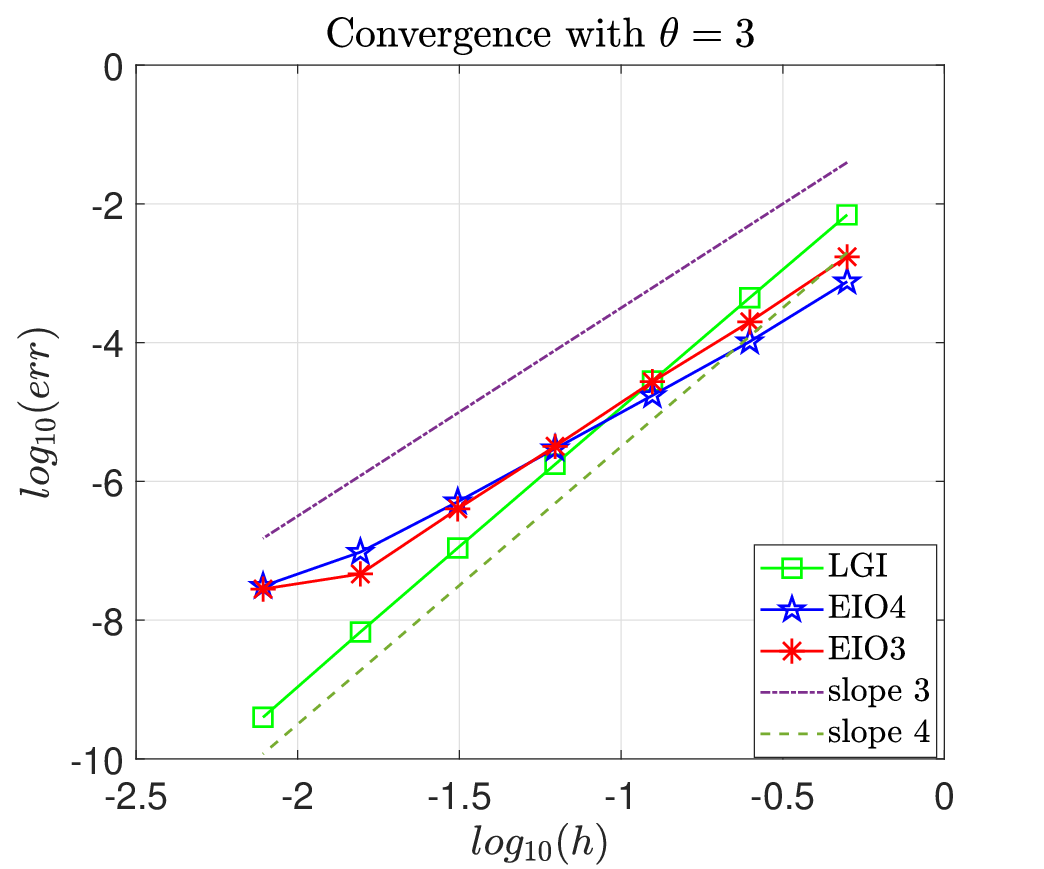}
\includegraphics[width=3.9cm,height=4.4cm]{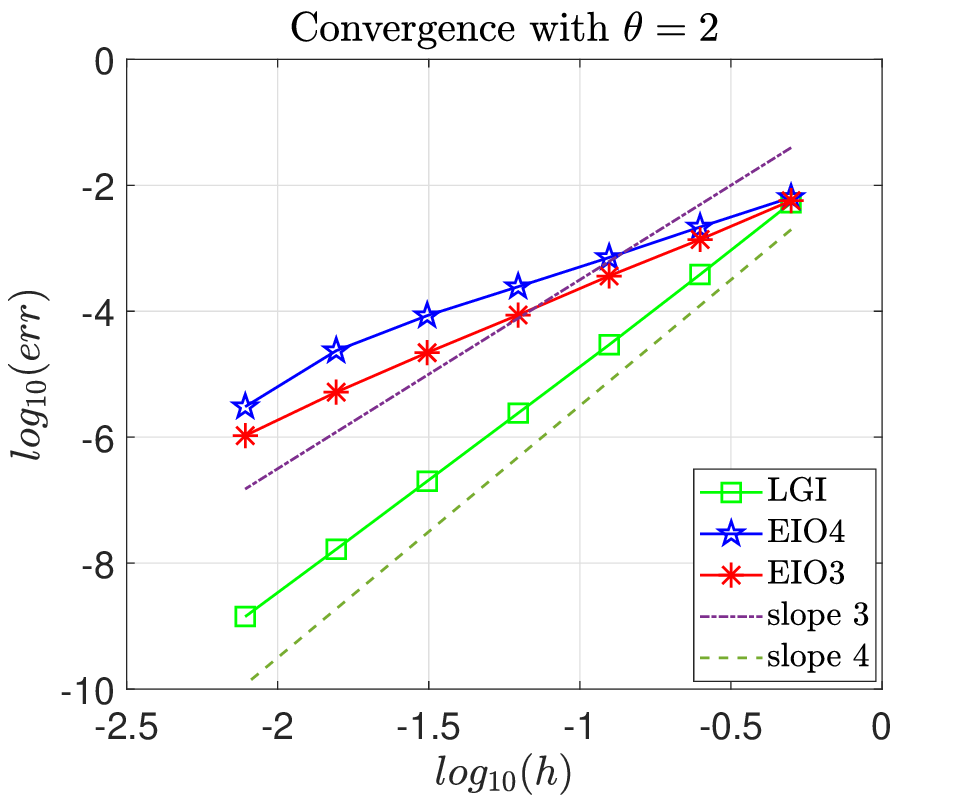}\\
\caption{Temporal error $err=\frac{\norm{u_n-u(t_n)}_{H^{1}}}{\norm{u(t_n)}_{H^{1}}}+\frac{\norm{v_n-v(t_n)}_{L^{2}}}{\norm{v(t_n)}_{L^{2}}}$ at $T=1$ with  initial data $H^{\theta}(\mathbb{T}) \times H^{\theta-1}(\mathbb{T})$ against $h $  with $h =1/2^k$, where $k=1,2,\ldots,7$.} \label{p1}
\end{figure}


    \begin{figure}[t!]
\centering
\includegraphics[width=3.9cm,height=4.4cm]{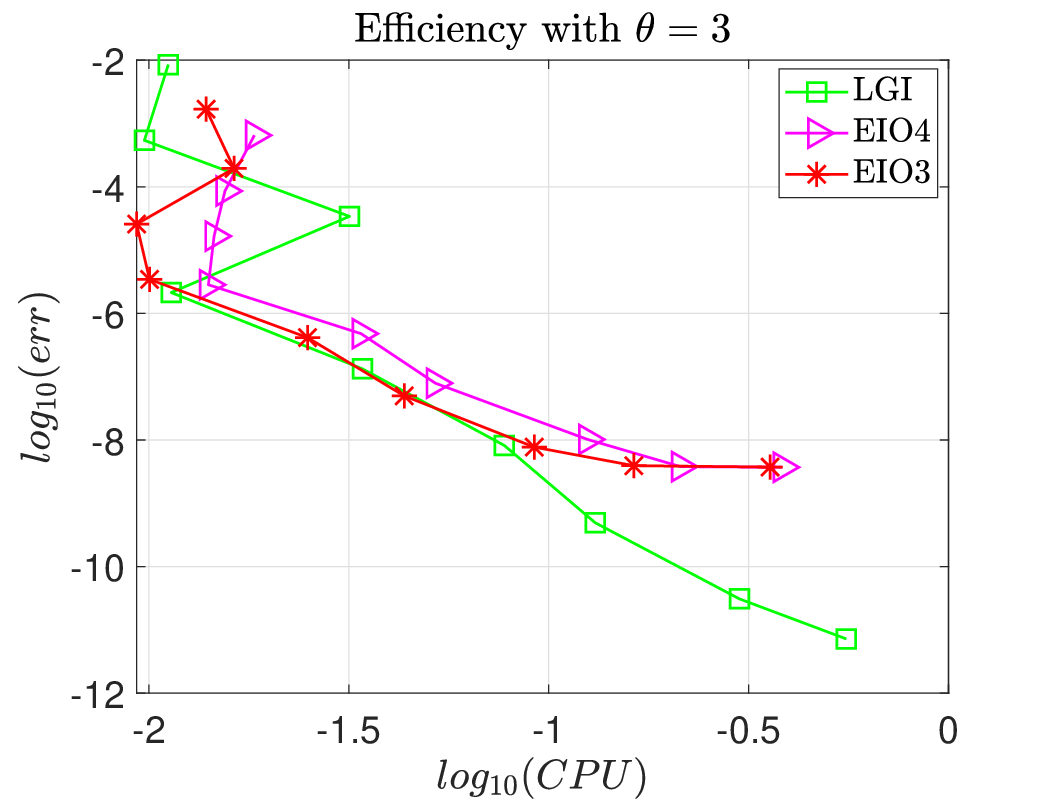}
\includegraphics[width=3.9cm,height=4.4cm]{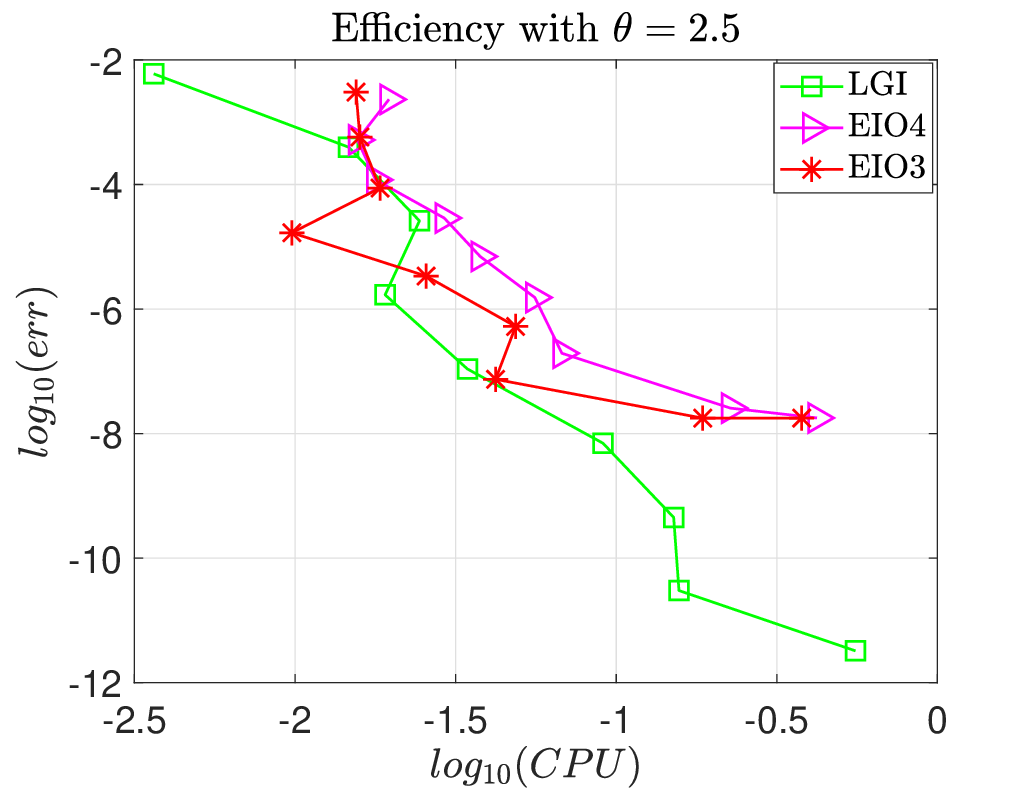}
\includegraphics[width=3.9cm,height=4.4cm]{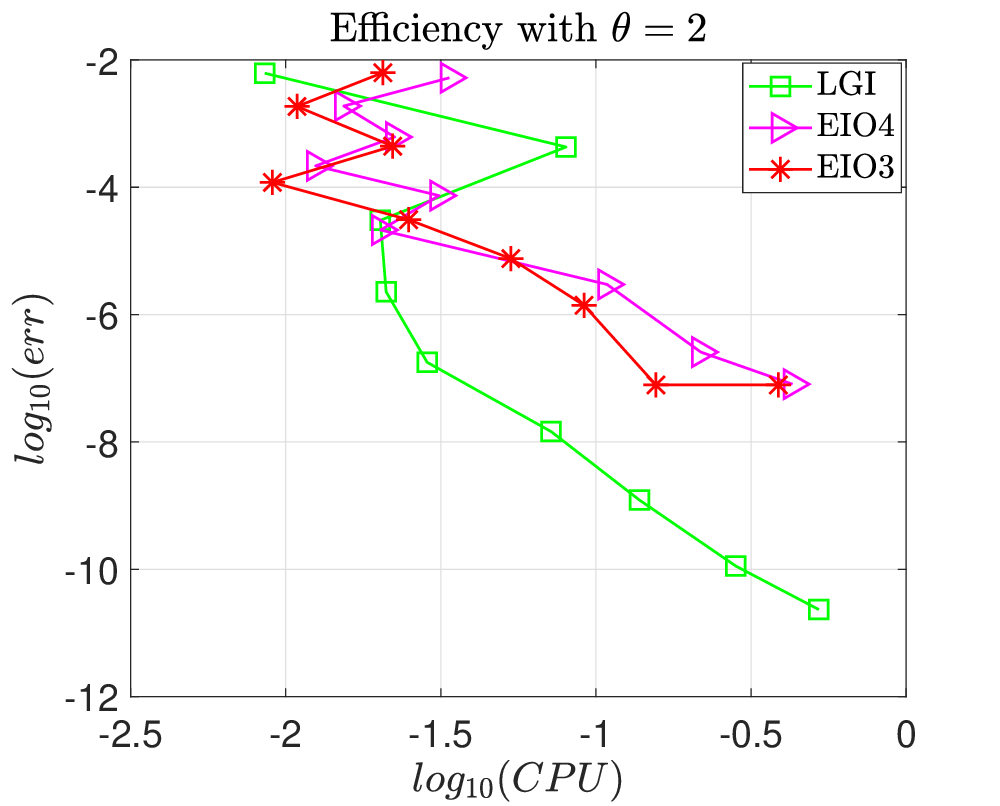}
\caption{Efficiency comparison: $err=\frac{\norm{u_n-u(t_n)}_{H^{1}}}{\norm{u(t_n)}_{H^{1}}}+\frac{\norm{v_n-v(t_n)}_{L^{2}}}{\norm{v(t_n)}_{L^{2}}}$ at $T=1$  against different CPU time produced by different  $h =1/2^k$, where $k=1,2,\ldots,7$.}\label{fig2}
\end{figure}

\section{Conclusions}\label{sec:5}
 In this paper,   a  low-regularity   trigonometric integrator was formulated and analysed  for
solving the nonlinear Klein-Gordon equation in the $d$-dimensional space with $d=1,2,3$.
Rigorous error estimates were given and the proposed integrator was shown to have
fourth-order time accuracy  in the energy space under a weak regularity condition.    A numerical experiment was carried out and the corresponding numerical results  were presented to  demonstrate  the behaviour of the new integrator  in comparison with   some well-known exponential integrators.

\section*{AI Usage Statement}
 No artificial intelligence tools are used in the whole process of the algorithm framework design, theoretical derivation, error analysis, numerical experiment scheme demonstration, and program coding of the fourth-order low-regularity method proposed in this paper. All algorithm improvement designs, mathematical analysis,  and numerical implementation codes are independently completed by the authors through manual research, deduction and programming. 

\section*{Acknowledgements}
We sincerely thank Professor Christian Lubich for his insightful comments, valuable suggestions, and rigorous guidance on this work. 
This work was supported by NSFC (12371403, 12501567, 12271426), the Youth Project (Category A) of Shaanxi Provincial Natural Science Basic Research Program (2026JC-JCQN-013), and the Shaanxi Fundamental Science Research Project for Mathematics and Physics (25JSY046).


\end{document}